\documentclass[runningheads,envcountsame,numbook]{svjour3}
\usepackage[width=12.1cm, left=4.5cm, height=20.75cm] {geometry}
\usepackage{cite}
\usepackage{amsmath}
\usepackage{amssymb}
\usepackage{mathptmx}
\usepackage{enumerate}
\usepackage{hyperref}
\usepackage[colorinlistoftodos,disable]{todonotes}
\usepackage[weather,alpine,misc]{ifsym}
\smartqed  % flush right qed marks, e.g. at end of proof

\newcommand{\eps}{\varepsilon}
\newcommand{\bx}{\bar x}
\newcommand{\by}{\bar y}

\newcommand{\iv}{^{-1} }

\newcommand {\R} {\mathbb R}
\newcommand {\N} {\mathbb N}

\newcommand {\Sp} {\mathbb S}
\newcommand {\gph} {{\rm gph}\,}%Graph
\newcommand {\dom} {{\rm dom}\,}
\newcommand {\epi} {{\rm epi}\,}

\newcommand {\co} {{\rm co}\,}

\newcommand {\Limsup} {\mathop{{\rm Lim\,sup}\,}}
\newcommand {\sd} {\partial}
\renewcommand{\iff}{$ \Leftrightarrow\ $}%iff

\def\nbh{neighbourhood}
\def\es{\emptyset}

\def\RHS{right-hand side}
\def\SVM{set-valued mapping}

\newcommand{\norm}[1]{\left\Vert#1\right\Vert}

\newcommand{\blue}[1]{\textcolor{blue}{#1}}
\newcommand{\cyan}[1]{\textcolor{cyan}{#1}}
\newcommand{\magenta}[1]{\textcolor{magenta}{#1}}
\newcommand{\red}[1]{\textcolor{red}{#1}}
\newcommand{\yellow}[1]{\textcolor{yellow}{#1}}

\newcommand{\ang}[1]{\left\langle #1 \right\rangle}
\newcommand{\qdtx}[1]{\quad\mbox{#1}\quad}

\newcounter{mycount}

\newcommand{\AK}[1]{\todo[inline]{AK {#1}}}
\newcommand{\rg}{{\rm rg}[F](\bx,\by)}

\newcommand{\tr}{{\rm tr}[A,B](\bx)}


\def\downto{{\raise 1pt \hbox{$\scriptstyle \,\searrow\,$}}}
\def\for{\hskip0.9pt|\hskip0.9pt}
\def\reg{\mathop{\rm reg}\nolimits}
\def\subreg{\mathop{\rm subreg}\nolimits}
\def\lip{\mathop{\rm lip}\nolimits}

\newcommand{\AD}[1]{\todo[inline,color=green!40]{AD {#1}}}
\def\rad{{\rm rad[SR]}}
\newcommand{\rgo}{{\rm rg}^\circ[F](\bx,\by)}
\newcommand{\rgd}{\underline{\rm rg}^\circ[F](\bx,\by)}
\newcommand{\rgh}{\overline{\rm rg}[F](\bx,\by)}
\newcommand{\rgdag}{\overline{\rm rg}^\circ[F](\bx,\by)}

\renewcommand{\tr}{{\rm tr}}
\newcommand{\HG}[1]{\todo[inline,color=cyan!40]{HG {#1}}}
\title{The Radius of Metric Subregularity
\thanks{Supported by the National Science Foundation (NSF) grant 156229; the Austrian
Science Fund (FWF) grants P26640-N25, P26132-N25 and P29190-N32; the Australian Research Council (ARC) grant DP160100854 and the Grant Agency of the Czech Republic (GACR) grants 17-04301S and 17-08182S.}
}

\author{
Asen~L.~Dontchev
\and
Helmut Gfrerer
\and
Alexander~Y. Kruger
\and
Ji{\v{r}}\'i~V. Outrata
}

\institute{
A.~L.~Dontchev \at
Department of Aerospace Engineering, The University of Michigan,
Ann Arbor, MI\\
\email{dontchev@umich.edu}
\and
H. Gfrerer \at
Institute of Computational Mathematics, Johannes Kepler University Linz, A-4040 Linz, Austria\\
\email{helmut.gfrerer@jku.at}
\and
A. Y.~Kruger (\Letter\,) \at
Centre for Informatics and Applied Optimization, Federation University Australia, PO Box 663 Ballarat VIC 3353, Australia\\
\email{a.kruger@federation.edu.au}
\and
J.~V. Outrata \at
Institute of Information Theory and Automation, Czech Academy of Science, Pod Vod\'arenskou
v\v{e}\v{z}\'i 4, 182 08 Prague, Czech Republic
\email{outrata@utia.cas.cz}
}
\dedication{Dedicated to Professor Alexander Ioffe on the occasion of his 80$^{th}$ birthday
}
\date{Received: date / Accepted: date}

\journalname{}

\begin{document}

\maketitle

\begin{abstract}
There is a basic paradigm,
called here the \emph{radius of well-posedness}, which quantifies the ``distance" from a given
well-posed problem to the set of ill-posed problems of the same kind. In variational analysis,
well-posedness is often understood as a regularity property, which is usually employed to
 measure the effect of perturbations and  approximations of a problem  on its solutions.
In this paper we focus on evaluating the radius of the property of metric subregularity which, in contrast to its siblings, metric regularity, strong regularity and strong subregularity, exhibits a more complicated  behavior under various perturbations. We consider three kinds of perturbations: by Lipschitz continuous functions, by semismooth functions, and by smooth functions, obtaining different expressions/bounds for the radius of subregularity, which involve generalized derivatives of set-valued mappings. We also obtain different expressions when using either Frobenius or Euclidean norm to measure the radius. As an application, we evaluate the radius of subregularity of a general constraint system. Examples illustrate the theoretical findings.
\end{abstract}



\AK{6/06/18.
Jiri's suggestions incorporated.
Lemmas 14 and 15 added.
They are referred to in the proof of Proposition 16.
}

\AK{26/05/18.
Helmut's suggestions (partially) incorporated.

1. Our main derivative-like object is now called \emph{primal-dual derivative}; see Subsection 2.3.

2. The statement of Proposition~\ref{P1} amended.

3. Proposition~\ref{P4} added.
It has allowed to remove $\widehat{\mathfrak{D}}^\diamond F(\bx,\by)$.
The proof of Proposition~\ref{P3} amended accordingly.

4. Overline and underline signs are used in the definitions of the auxiliary regularity constants.
I keep the original notations for the `main' constants $\rg$ and $\rgo$.

5. The statement of Proposition~\ref{P3} amended.

Besides, a paragraph added in the Introduction with the references to most recent (since 2015) publications on metric subregularity, including the paper by Mar\'echal suggested by Asen.
Should the list be shortened?
}
\AK{19/05/18.
1. Jiri's notes incorporated.

2. A description of the structure os the paper added at the end of the Introduction.

3. Line numbering added to simplify communication.}

\AK{18/05/18.
Asen's notes incorporated.
In particular, Lemma~4 removed.
In some cases, Asen's or my comments are included in the text.
}

\AK{15/04/18.

1) Following the recent instance of confusion with the notation, $\widehat{\mathcal{D}}F(\bx,\by)$ has been changed to $\widehat{\mathfrak{D}}F(\bx,\by)$.
I still want to keep some form of $D$ in the notation as it is about derivatives.
The definition \eqref{DF2} of $\widehat{\mathfrak{D}}^\circ F(\bx,\by)$ has been changed to make it also an image set like \eqref{DF} and \eqref{DF1}.
This raises a question about the relationship between \eqref{DF1} and \eqref{DF2}, which includes possible `reversing' of Lemma~4 -- see the question on page~11.

\blue{AK 17/05/18.
As suggested by Asen, Lemma~4 has been removed.
The assertion is now a part of Proposition~4, and the question is now on page 13.}

2) Helmut's example of the absence of robustness of the (strengthened) subregularity property has been added as another subsection in Section 5.

3) Several subsection titles have been added to improve the structure.

4) The two Asen's examples in the Introduction have been made Examples~\ref{E1} and \ref{E2} and are now cited in the main text.

5) The statement of the main radius theorem has been reversed to the previous version with the additional estimates moved to Corollary 8.

\blue{AK 17/05/18.
Now Corollary 7.}

6) The concluding part of Subsection 3.1 has been expanded with some new comments added.

7) The statement of Proposition 10 has been slightly amended to improve its readability.

8) Section 5 expanded.
}

\AK{7/04/18.

1) The issues picked up by Jiri are fixed.

2) Section 5 (Example) appended by a discussion of ways of computing regularity constants by solving optimization problems, based on Jiri's handwritten notes.
}

\AK{30/03/18.

1) Amendments are made to Section 5 (Example) following comments from Jiri.

2) The regularity constants are subdivided into two groups: the main two: $\rg$ and $\rgo$, and modifications of $\rg$.
The modifications are now denoted in a uniform way: $\rgd$, $\rgh$ and $\rgdag$.

3) Amendments are made in the Introduction, explaining why the conclusions of the radius theorem fall into the general radius pattern.
See also the updated Remark 6.

4) $\rgo$ is added in the radius theorem as another upper bound for $\rad_{Lip}F(\bx\for\by)$.
}

\AK{25/03/18.

1) Section 5 (Example) has been rewritten using the computations performed by Jiri.
All regularity constants have been computed.

2) To prepare the ground for the changes in Section 5, computing constant $\rgh$ has been added to all statements in Section 4.
}
\AK{13/03/18.

1) Following Jiri's advice (in a Skype session), two clarifications have been added to (my version of) Helmut's proof, currently at the end of pages 12 and 13.

2) Again based on Jiri's handwritten notes, the issue with the signs in Proposition~\ref{P2} has been fixed.
The statement looks nice again.
The issue was caused by my incorrect application of \cite[formula (2.4)]{GfrOut16.2}; I apologize for that.
}
\AK{10/03/18.

The new addition to the radius theorem by Helmut looks great.
I've done some revision of the proof, but apart from fixing a few typos, it is not really an improvement of the proof; it is more like my attempt to understand its logic `on the go'.
A few things are left in red for checking and approval.

Since the proof of the theorem grew longer, I've added a bit of a structure by inserting underlined headings in the beginning of each part of the proof.

Following the email from Jiri, I've made a few changes to the signs in the proof to make the conclusion more straightforward.
These changes don't affect the conclusion.
Neither they affect the issue with the signs in Section~4.
It looks like there was a mistake in the version DKO7 and earlier ones.
}
\HG{4/03/18.
1) Upper bound $\rgh$ for $\rad_{Lip}$ added in the radius theorem: It is built by $\norm{u^\ast}_*+\norm{v}$ and compared with the lower bound $\max\{\norm{u^\ast}_*,\norm{v}\}$ the difference is a factor less than or equal to 2. The result looks beautiful but the proof takes nearly 2 pages.
2) Some comments added.
}
\AK{4/03/18.

1) Most of the `colours' removed, assuming that the corresponding earlier changes have been (silently) approved by everybody.
A few `colourful' places still remain where I do want to attract attention and get comments.
The new changes are highlighted in red as always.

2) The table of contents added temporarily on page 4 to see the overall structure of the paper `at a glance'.

3) Several subsection headings added, possibly also temporarily, to better structure the paper.
We can either remove them later or, on the contrary, add more subsections.

4) The pieces about `derivatives' and `regularity constants' moved from various locations in the main section into Preliminaries and expanded.
This seems to be the most questionable part.
Could be expanded further by including the definitions of the critical set(s) and directional coderivatives and appropriate discussions.

5) Signs changed in some definitions in subsection 2.2 for consistency.
I am again a bit unsure if everything is still correct.
This might have consequences throughout the paper.

6) Following Jiri's suggestion another `regularity' constant \eqref{rgdag} introduced corresponding to the Euclidean norm.

7) Section 4 updated. There seems to be a problem with the signs there.
}

\AK{1/03/18.

1) A few amendments to the Introduction and Preliminaries.

2) The definitions of the radii for several classes of perturbations are merged into a single formula.

3) Amendments to the presentation in Section 3.

4) Corollary 8 added.

\blue{AK 17/05/18.
As requested by Asen, Corollary 8 has been removed.}
}

\AD{27.02.18. Asen's  updates in cyan}

\AK{22/02/18.

I have made some amendments.

1) Lemma~\ref{L2} shortened: the Clarke Jacobean has been moved to the proof of the main theorem.

2) New mixed primal-dual derivative-like object introduced together with the corresponding `regularity' constants.
The definitions and notations need to be checked and discussed.

3) The two radius theorems are merged into a single one.

4) Several facts from the previous version are formulated as separate statements: Proposition~3 and Lemma~4.

\blue{AK 17/05/18.
As suggested by Asen, Lemma~4 has been removed.
The assertion is now a part of Proposition~4.}

5) Propositions~5 and 6 amended to comply with the definition of the new derivative-like and regularity objects.

I have not touched the rest of the paper.}

\HG{10/02/18.

The main changes were done in section 2. Ji\v{r}i and I observed that the radius theorem is not correct when taking the infimum of $\norm{B}$ over $\partial_C h(\bx)$: One has to take the supremum. But $\lip(h;\bx)=\sup\{\norm{B}\mid B\in \partial_Ch(\bx)\}$ and $\lip(h;\bx)=\norm{\nabla h(\bx)}$ in the $C^1$ case and therefore we used $\lip(h;\bx)$ in the definitions of the different radii.

We now consider three different classes of perturbations: Lipschitzian (abbreviated by $Lip$, semismooth Lipschitzian ($ss$) and smooth perturbations ($C^1$). A lower bound for the radius of metric subregularity in the Lipschitzian case is added and this lower bound is compared with the radii for metric regularity (abbreviated MR) and strong metric subregularity (SSR).

We further tried to eliminate the matrix $B$ from the computations in the semismooth (smooth) case. The resulting formulas are easier and allow a comparison with the Lipschitzian case but the prize we have to pay for is that the resulting formula is not exact and provides only lower and upper bounds which differ by a factor of at most $\sqrt{2}$ (so these bounds are rather sharp).

In the current version we allow any norms in $\R^m$ and $\R^n$ and, in fact, for every pair of norms we get different radii. The obtained results are only correct if we use for the dual objects the corresponding dual norms. This is not explicitly mentioned and we also use the notation $\R^m$ and $\R^n$ for the dual spaces which, in my opinion, is very confusing. Maybe it would be better to use only the Euclidean norm, however this would mean that we cannot use the wonderful ideas of Alexander for computing the radius in the example with respect to the maximum norm.
}

\AK{7/02/18.

The recent contributions from Jiri and Asen implemented.
Lemma~\ref{L2} added.
The proof of Theorem~\ref{ThRadLip} is now based on Lemma~\ref{L2}.}

\AK{5/02/18.

The recent contributions from Jiri implemented.
Sections 3 and 4 swapped.}

\AK{25/01/18.

The recent contributions from Jiri and Asen implemented.

A short section added based on Jiri's notes.
It needs to be checked.

Besides this last section, a few other things (highlighted in red) in the preceding sections need to be checked and confirmed, especially the new statement of Theorem~\ref{ThRadLip} and related definitions and notations.}

\AK{15/01/18.

The recent contributions from Jiri and Asen implemented.

1) The former Theorem 1 (Radius theorem for (strong) metric (sub)regularity) removed. Why?

2) The main Theorem~\ref{ThRadLip} reformulated using two families of perturbations and special notations for the radii.
%(Should we write `radii'?)
%\blue{AK25/01/18: Changed to radii everywhere.}
}

\todo[inline]{AK 21/01/17.

Environment prepared for the convenience of the joint work on the source file.
This should also simplify applying journal styles in the future.

1) Page size -- A4.

2) Standard page parameters (margins) given by package {\tt fullpage}.

\blue{AK 14/01/18.
{\tt fullpage} is replaced by more realistic {\tt geometry}.}

3) Standard maths notations and theorem environments given by the AMS packages {\tt amsmath} and {\tt amsthm}.

4) Asen's  definitions replaced by mine.
They are in two separate files {\tt macro.tex} and {\tt macrothm.tex} (not for editing!) which should be placed in the same folder with this source file.
A few Asen's definitions which are used in the current draft and not covered by mine are now in the separate file {\tt macro-AD.tex} (feel free to edit) which should also be placed in the same folder.

The only conflict between the two sets of definitions is with the macro {\tt $\backslash$iff}.
I have replaced all cases of {\tt $\backslash$iff} in the manuscript by ` \iff'.
Please avoid using {\tt $\backslash$iff}.

5) The bibliography is now created automatically by BibTeX from my databases (mostly based on MathSciNet) which I am happy to share if you wish.
Otherwise, just do not run BibTeX.
LaTeX will do everything right as long as the file {\tt BBL} is in the same directory.
When adding new references, I suggest you keep using the standard {\tt thebibliography} environment as previously.
I will then merge the two lists of bibliography.

6) When making changes in the source file, please identify them with colours using appropriate `colour commands' \red{\tt $\backslash$red}, \blue{\tt $\backslash$blue}, \magenta{\tt $\backslash$magenta}, \cyan{\tt $\backslash$cyan} or \yellow{\tt $\backslash$yellow}.
Choose your favourite.
\medskip

I have made a few minor changes in the text in general (like e.g. the numbering of the formulas and theorems) and in the proof of Theorem~\ref{ThRadLip}.

The major change is removing the `hat' from the notation of the Fr\'echet normal cone.
}

%\setcounter{tocdepth}{2}
%\tableofcontents

\section{Introduction}\label{intro}

According to the classical definition of  Hadamard, a mathematical problem is well-posed when it has a unique solution which  is a continuous function of the data of the problem.
Establishing the well-posedness is a basic task, but there are other questions around it such as how ``robust" the well-posedness property is under perturbations, or how  ``far" from a given well-posed problem the ill-posed
problems are.
The formalization of the  latter question leads to the concept of the {\em radius of well-posedness},  which quantifies the distance from a given
well-posed problem to the set of ill-posed problems of the same kind.

To be specific, consider the problem of solving
 the linear equation $Ax=b$, where $A$ is an $n\times n$ matrix and $b\in\R^n$. This problem is well-posed in the sense of Hadamard exactly when the matrix $A$ is nonsingular. The radius
of well-posedness of this problem is well known, thanks to  the \emph{Eckart--Young theorem} \cite{EckYou36}, which says the following:
for any nonsingular $n \times n$ matrix $A$,
\begin{gather}\label{ek}
\inf_{B \in L(\R^n,\R^n)}\{\|B\| \mid A+B \mbox{ singular}\}
= \frac{1}{\|A^{-1}\|},
\end{gather}
where $L(\R^n,\R^m)$ denotes the set of $n\times m$ matrices, and $\|\cdot\|$ is the usual operator norm.
In numerical linear algebra this
theorem  is intimately  connected with the  \emph{conditioning} of the matrix $A$.
Namely, the expression on the right-hand side
of \eqref{ek} is the reciprocal of the \emph{absolute condition number} of $A$; dividing by
$\|A\|$ would give us a similar expression for the \emph{relative condition number}.
Thus, the radius equality \eqref{ek} is in line  with the  idea of conditioning; the farther a matrix is from the set of singular matrices, the
better its conditioning is.
The reader can find a  broad coverage of the mathematics around condition numbers and conditioning in the  monograph \cite{BurCuc13}.

\if{
\AD{sentence deleted: The ``good behavior" of a problem is more generally understood as a regularity property which describes desirable features of the solutions.}
}\fi
A far reaching generalization of the Eckart--Young theorem was proved in \cite{DonLewRoc03} for the property of \emph{metric regularity} of a set-valued mapping $F$ acting generally between metric spaces, which is the same as  nonsingularity when $F$ is a square matrix.
This generalization was later extended in \cite{DonRoc04} to the properties of \emph{strong metric regularity} and
\emph{strong metric subregularity}, see also \cite[Section~6A]{DonRoc14}.
In this paper we deal with the radius of \emph{metric subregularity}, a property which turns out to be quite different from its siblings.
%AK26/02/18
%mentioned above.

We proceed
%AK28/02/18
%here
now
with the definitions of these properties; more details regarding  the notation and the definitions used   in the paper  are given in Section~\ref{prelim}.

%AK26/02/18
%Recall that
A set-valued mapping $F$ acting from $\R^n$ to $\R^m$  is said to be  {\it metrically regular\/} at $\bx$ for $\by$  if $(\bx,\by)\in\gph F$
\if{
\AD{deleted: , the graph of $F$ is locally closed around $(\bx, \by)$,}
}\fi
and there exists a number $\kappa \in [0,+\infty)$ together  with
neighborhoods $U$ of $\bx$ and $V$ of $\by$ such that
\begin{gather}\label{mr}
d(x, F^{-1}(y)) \,\leq\, \kappa d(y, F(x))
                 \mbox{ for all } x\in U,\;y\in V.
\end{gather}
Here $d(x, C)$ is the distance from a point $x$
%AK26/02/18
%\in \R^n$
to a set $C
%\subset \R^n
$: $d(x, C) = \inf_{y \in C}\|x-y\|$.
The infimum of the set of values $\kappa $ for which \eqref{mr} holds is called the
{\it modulus of metric regularity\/}, denoted by $\reg(F;\bx\for\by)$.
%The absence of metric regularity is signaled by $\reg(F;\bx\for\by)= +\infty$.
\if{
\AD{deleted A linear and bounded mapping acting between Banach spaces is metrically regular if and only if it is surjective.}
}\fi
A mapping $F$ is metrically regular at $\bx$ for $\by$ if and only if its inverse $F^{-1}$ has the \emph{Aubin property} at $\by$ for $\bx$,
a property which in the single-valued case reduces to the Lipschits continuity.

\if{
\AD{deleted:; that is,  there exist  neighborhood  $U$ of $\bx$ and $V$ of $\by$ and
a constant $\kappa$ such that $$
F^{-1}(y) \cap U \subset F^{-1}(y') + \kappa \|y-y'\| \quad \mbox{for all } y, y' \in V.
$$
 The infimum of the  values $\kappa$ for which \eqref{aub} holds equals $\,\reg (F;\bx\for\by)$.}
}\fi

A mapping  $F:\R^n\rightrightarrows \R^m$ with $(\bx,\by) \in \gph F$ is said to have a \emph{single-valued  localization} around $\bx$ for $\by$ if there exist neighborhoods
$U$ of $\bx$ and $V$ of $\by$ such that the  truncated
mapping $U\ni x \mapsto F(x)\cap V$ is single-valued, a  function on $U$.

If the inverse
$F^{-1}$ of a mapping $F$  has a localization at $\by$ for $\bx$ which is Lipschitz continuous, then   $F$ is said to be {\em strongly metrically regular},
or simply \emph{strongly regular}; in this case $F$ is automatically metrically regular at $\bx$ for $\by$ and the Lipschitz modulus of the localization at $\by$ equals $\reg (F;\bx\for\by)$.

If we fix $y$ in \eqref{mr} at its reference value $\by$, we obtain the property of {\em metric subregularity}, which we sometimes call simply \emph{subregularity}. Specifically,
a mapping $F:\R^n \rightrightarrows \R^m$  is said to be metrically subregular at $\bx$
for $\by$ if $(\bx,\by)\in\gph F$ and
there exists a number $\kappa \in [0,+\infty)$ together  with a
neighborhood $U$ of $\bx$ such that
\begin{gather}\label{msr}
d(x, F^{-1}(\by)) \,\leq\, \kappa d(\by, F(x))
                 \mbox{ for all } x\in U.
\end{gather}
The infimum of the set of values $\kappa$ for which \eqref{msr} holds is called the
{\it modulus of metric subregularity\/}, denoted by $\subreg (F;\bx\for\by)$.
%The absence of metric subregularity is signaled by $\subreg(F;\bx\for\by)=+\infty$.
A mapping $F$ is metrically subregular at $\bx$ for $\by$ if and only if its inverse $F^{-1}$ is \emph{calm} at $\by$ for $\bx$,
a property which corresponds to the Aubin continuity with one of the variables fixed.

\if{
\AD{deleted:
In general, a set-valued mapping $S:\R^m \rightrightarrows \R^n$ is said to be calm at $\by$ for $\bx$ whenever $(\by,\bx) \in \gph S$ and there exist a neighborhood  $U$ of $\bx$ and
a constant $\kappa$ such that
$$
S(y) \cap U \subset S(\by) + \kappa \|y-\by\| \quad \mbox{for all } y \in \R^m.
$$
}}\fi

A mapping $F$ is said to be {\em strongly metrically subregular}, or simply \emph{strongly subregular} at $\bx$ for $\by$ if $F$ is metrically subregular at $\bx$ for $\by$
and in addition  $\bx$ is an isolated point in $F^{-1}(\by)\cap U$.
In this case, $F^{-1}$ has the {\em isolated  calmness} property at $\by$ for $\bx$.

%AK28/02/18
%When $f$ is a function, we write, with some abuse of notation, $\reg (f;\bx) = \lip (f^{-1}; f(\bx))$.
If $f$ is a (single-valued) function, we write, with some abuse of notation, $\reg(f;\bx)$ and $\subreg(f;\bx)$ instead of $\reg(f;\bx\for f(\bx))$ and $\subreg(f;\bx\for f(\bx))$, respectively.

Clearly, the above definitions of  regularity properties can be extended in a straightforward manner to general metric spaces.

All the above concepts
%AK26/05/18
have been well studied.
They
are discussed in detail in \cite{RocWet98,KlaKum02,Mor06.1,DonRoc14,Iof17}.
The metric subregularity, which is the main object of study in the current paper, is implicitly present  already in the pioneering work by Graves \cite{Gra50}, as shown in \cite[Section~5D]{DonRoc14}.
This property plays a major role in deriving
the Lagrange multiplier rule in its various forms, see e.g. \cite[Section~2.1]{KlaKum02}.
For the most recent developments in research on metric subregularity, we refer the readers to \cite{Kru15,Kru15.2,NgaPha15, DurStr16,MarCor16,Ude16,Zhe16,ZheZhu16, Iof17,NgaTroTin17, CibDonKru18,Mar18}.

%In this paper we deal with the following question: given a mapping $F$ with one of the properties described above, how much could one perturb the mapping before it looses its regularity property. Following \cite{DonLewRoc03},
%we call the associated quantity {\em the radius of regularity.}
%AK28/02/18
%If a mapping does not possess a certain regularity property at the reference point, we will adopt the natural convention that the corresponding radius equals 0.

It turns out that the Eckard--Young equality \eqref{ek} is a  special case of a general paradigm
which can be described as
\begin{gather}\label{radreg}
{\rm rad} = \frac{1}{\reg},
\end{gather}
where rad is the appropriately defined radius of the considered regularity property, and $\reg$ is the  modulus
of this property.
This paradigm was first established in
%AK28/02/18
%the paper
\cite{DonLewRoc03} for the property of metric regularity.
Specifically, it was established that if a mapping $F:\R^n \rightrightarrows \R^m$ is metrically regular at $\bx$ for $\by$, then
% we have
\begin{multline}\label{rad}
{\rm rad[MR]}F(\bx\for\by):=\inf_{B\in L(\R^n,\R^m)} \big\{\|B\| \mid
\\
F+B\mbox{ is not metrically regular at $\bx$ for }  \by+B\bx \big\}
= \frac{1}{\reg(F; \bx\for\by)}.
\end{multline}
Moreover, the
%AK26/02/18
% infimum on the left side of \eqref{rad} is unchanged if taken with respect to matrices  $B$ of rank one, but also remains unchanged when
equality remains true if the infimum is taken with respect to all matrices $B$ of rank one, or
the class of perturbations is enlarged to the
%set
family
of functions $h: \R^n\to\R^m$ that are Lipschitz continuous around $\bx$, with $\|B\|$ replaced by the Lipschitz modulus $\lip(h;\bx)$.
That is, the radius of metric regularity is the same for all perturbations $h$ ranging from Lipschitz continuous functions to linear mappings of rank one.

Subsequently, in \cite{DonRoc04} this  radius equality was shown to hold in the same form for the properties of strong regularity and strong metric subregularity. Specifically,   if a mapping $F:\R^n \rightrightarrows \R^m$ is strongly  regular or strongly subregular  at $\bx$ for $\by$, respectively, then  the equality \eqref{rad} holds with ``not metrically regular" replaced  by ``not strongly regular" or ``not strongly subregular", respectively, and in the second case $\reg(F; \bx\for\by)$ on the right side is replaced by $\subreg(F; \bx\for\by)$.

%\AD{ deleted: \begin{multline}\label{subrad}
%{\rm rad[SSR]}F(\bx\for\by):=\inf_{h \in {\cal L}(\bx)} \big\{ \lip (h;\bx) \mid   \\  F+h
%\mbox{is not strongly  subregular at $\bx$ for}  \by+h(\bx)\big\}
%= \frac{1}{\subreg(F; \bx\for\by)}.
%\end{multline} }

%AK30/03/18
In some situations it is more convenient to work with the reciprocal of the regularity modulus reg.  We denote this reciprocal by rg and then equality \eqref{radreg} becomes
\begin{gather}\label{radreg+}
{\rm rad} = {\rm rg}.
\end{gather}
%with the reciprocal rg $=\reg\iv$
%AK4/06/18
%playing the role of the `modulus' of
%\red{substituting the modulus when testing}
%the respective regularity property.
In the case of the conventional metric regularity, rg corresponds to the {\em modulus of surjection} `sur' used by Ioffe \cite{Iof17}; see also other examples in \cite{Kru15,KruTha15,KruLukTha17}.
This notation is in agreement with the natural convention, which we adopt here, that if a mapping does not possess a certain regularity property,  then the regularity modulus equals $+\infty$ and  the corresponding radius equals 0.

It turns out, however, that
the (not strong) metric subregularity does not obey the radius paradigm,
at least
in the form \eqref{radreg}
%AK30/03/18
{or \eqref{radreg+}}.
This effect  was first noted in \cite{DonRoc04} and also discussed in \cite[Section~6A]{DonRoc14}.

\if{
\AD{deleted: In order to fully present the motivation behind the present paper, we briefly
describe  the situation on an example.
 First, recall that calmness, as defined in \eqref{cm},  is a local version of the so-called \emph{outer Lipschitz continuity} at $\by$,  which, for a mapping $S:\R^m \rightrightarrows \R^n$ with $(\by, \bx) \in \gph S$ is defined as in
\eqref{cm}
with $U = \R^n$.
}}\fi

%AK14/04/18
\begin{example}\label{E1}
%AK28/02/18
%Indeed,
%Consider
%first
%for instance,
%the polyhedral case.
By
%AK16/05/18 I am not sure
\todo{I would put `the'}
%the
{a}
fundamental result of Robinson \cite{Rob81}, every polyhedral mapping, that is, a mapping whose graph is the union of finitely many polyhedral convex sets, is outer Lipschitz continuous around every point in its domain.
Hence, inasmuch  outer Lipschitz continuity of the inverse implies metric subregularity,
%AK28/02/18
%we obtain that
every polyhedral mapping $F$  is metrically subregular at any  $\bx$ for any $\by$ such that $(\bx,\by)\in\gph F$.
It is elementary to observe that the sum of any polyhedral mapping and a linear mapping is again polyhedral.
Hence,
%AK26/02/18
%we have that,
if $F:\R^n\rightrightarrows\R^m$ is a polyhedral mapping and $(\bx,\by)\in\gph F$,
%is any point in its graph,
then
\begin{gather}\label{subrad}
%{\rm rad[
%AK26/02/18
%SMR
%\red{SR}
%]}\cyan{F(\bx\for\by)}:=
\inf_{B \in L(\R^n, \R^m)} \big\{\|B\| \mid
F+B\mbox{ is not metrically subregular at $\bx$ for }  \by+B\bx \big\}
= +\infty.
\end{gather}
Clearly, the quantity $\subreg(F;\bx\for\by)$ could be anything; thus the equality \eqref{radreg}  does not hold in general
%AK28/02/18
%(for the corresponding quantities) if linear perturbations are considered.
for polyhedral mappings.
\qed\end{example}

%AK15/04/18
\begin{example}\label{E2}
%As a further example,
Consider the zero function $f:\R\to\R$, that is $f(x) = 0$ for all $x\in\R$.
Then $f^{-1}(0) = \R$ and
$f^{-1}(y) = \emptyset$ for all $y \neq 0$.
Thus, the zero  mapping is metrically subregular at any $\bx$ for $0$, and the subregularity modulus is of course zero.
The function $h(x) = x^2$ is  Lipschitz continuous around $\bx = 0$ with Lipschitz modulus zero, but the mapping
$(f+ h)(x) = x^2$ is not metrically subregular at $0$ for $0$. Hence, the radius of metric subregularity of the zero mapping
%AK26/02/18
with respect to smooth perturbations
is zero, but this does not fall into the pattern of \eqref{radreg}.
Also note that the zero function is a polyhedral mapping, hence, in the light of the preceding example, its radius for linear perturbations is $+\infty$, while when we change to quadratic perturbations and use the Lipschitz modulus to measure the radius, it becomes zero.
\qed\end{example}

\if{\AK{3/03/18.
The above two examples to be formulated as numbered Examples and commented on in the text.}}\fi

Note that there are four components  involved in a radius equality \eqref{rad}: a regularity property, the basic underlying mapping $F$, the mapping $B$ representing the perturbations, and the ``size" of the perturbation, which in this case is measured by the norm of $B$. In this paper we consider the metric subregularity property, {for which} the basic mapping $F$ will be a set-valued mapping with closed graph. The perturbations will be represented by the following three classes of functions: Lipschitz continuous functions, semismooth functions and continuously differentiable ($C^1$) functions, all around/at the reference point. For all the three classes we will use
 the Lipschitz modulus at the reference point  as a measure of  the size of the perturbation. Note that for the second class the Lipschitz modulus can be expressed in terms of Clarke's generalized Jacobian, while for $C^1$ functions this would be the norm of the derivative at the reference point.

\if{
\AD{delete As an example, let $F$ be the linear function $F(x)= x$ acting from $\R$ to $\R$. If the perturbation
$h$ is another {\em linear} function, then the radius of subregularity will be $+\infty$. More generally, if $F$ is a polyhedral mapping and $h$ is a linear function then, by the   result by Robinson cited in preceding lines, the radius of subregularity will also be $+\infty$.
If, however, the perturbation is allowed to be a nonlinear smooth  function,
then the radius will be different.
Indeed, consider the preceding one-dimensional example $F(x) = x$ and
suppose that $h$ is a $C^1$ function, where the radius is measured by the norm of the derivative at $0$.
That is, we would like to find the infimum of $|h'(0)|$ over all smooth functions at $0$ such that $x + h(x)$ is not subregular.
It is not difficult to see that the radius is 1 and is attained for $h(x) = -x + x^2$.
Note that the \eqref{radreg} equality holds
%AK7/02/18
%for
in
this example
(as well as for $F(x) = \kappa x$
%AK7/02/18
%for
with
any real $\kappa$).}
}\fi

\if{
\AK{26/02/18.
The considerations below are partially duplicated in Proposition~\ref{P7}.}

Observe that, since both metric regularity and strong subregularity imply metric subregularity,
from \eqref{rad} and \eqref{subrad} we obtain the following lower bounds.
If $F$ is metrically regular at $\bx$ for $\by$, then
\begin{multline}\label{subrad1}
\rad F(\bx\for\by):=\inf_{h \in
\red{{\cal L}(\bx)}} \big\{ \lip (h;\bx) \mid
\\
F+h\mbox{ is not subregular at $\bx$ for }  \by+h(\bx) \big\}\geq  \frac{1}{\reg(F; \bx\for\by)},
\end{multline}
and if $F$ is strongly subregular  at $\bx$ for $\by$, then
\begin{gather}\label{subrad2}
\rad F(\bx\for\by)\geq  \frac{1}{\subreg(F; \bx\for\by)}.
\end{gather}
Since strong regularity implies both   metric regularity and strong subregularity, and also
$\subreg(F; \bx\for\by) \leq \reg(F; \bx\for\by)$, from \eqref{subrad1} and \eqref{subrad2}
 we obtain the following:  if $F$ is strongly regular  at $\bx$ for $\by$, then
%AK23/01/18
%$$
%\inf_{h \in {\cal L}(\bx)} \left\{ \lip (h;\bx) \mid F+h
%\mbox{is not  subregular at $\bx$ for  $\by+h(\bx)$} \right\}
%\geq  \frac{1}{\subreg(F; \bx\for\by)}.  $$
inequality \eqref{subrad2} holds.
}\fi

\if{In this paper, we
%AK30/03/18
%are going to
establish a radius formula of the type \eqref{radreg+} and certain radius estimates for several classes of perturbations for the subregularity property, using (sub)regularity constants (moduli), different from $\subreg(F;\bx\for\by)$.
%AK16/05/18
Such a constant being strictly positive ensures a stronger version of subregularity (which is still in general weaker than the conventional strong subregularity).}\fi
\AK{16/05/18.
I am keeping the last paragraph for now.

I think we have unearthed a new regularity property which can become important also outside of the radius stuff.
Just leaving a bridge to possible future work.}

\if{\AK{30/03/18.
Is it possible to show that in the case of strong subregularity all our constants are strictly positive?
Are they going to reduce to (the reciprocal of) $\subreg(F;\bx\for\by)$?}}\fi

The next Section~\ref{prelim} provides some preliminary material used throughout the paper.
This includes basic notation and general conventions, definitions of the three classes of perturbations studied in the paper and corresponding radii, and a certain new primal-dual derivative which gives rise to a collection of `regularity constants' used in the radius estimates.
In Section~\ref{main}, we establish
%the main result of the paper -- the radius theorem.
%It provides
lower and upper bounds for the radius of metric subregularity for Lipschitzian perturbations and the exact radius formula for the other classes of perturbations.
%Some simplifications of the formulas involved in the radius theorem are given for the case of the Euclidean norm.
%They employ a
The case when the size of the perturbation is measured by the
Frobenius norm on the space of matrices
is also discussed.
Section~\ref{S4} is devoted to applications to constraint systems,
%AK27/06/18
{while the last Section~\ref{S5} identifies possible directions for future research.}

{\section{Preliminaries}\label{prelim}}

\if{
\AD{the entire section is new but also uses some material defined previously}
}\fi

\subsection{Notation and general conventions}

Throughout we consider mappings acting
between finite dimensional spaces $\R^n$ and $\R^m$.
The spaces are assumed equipped with arbitrary norms denoted by the same symbol $\|\cdot\|$.
We usually keep the same notation for the duals of $\R^n$ and $\R^m$.
However, in some situations when this can cause confusion, we write explicitly $(\R^n)^*$ and $(\R^m)^*$.
The corresponding dual norms are denoted $\|\cdot\|_*$.
%AK16/05/18
%\red{$m\times n$ matrices are identified with linear operators acting from $\R^n$ to $\R^m$.
%Unless explicitly stated otherwise, the operator norm is used on the space of such matrices: $\norm{B}:=\sup\{\norm{Bx}\mid \norm{x}=1\}$.}
%\AK{16/05/18.
%Why deleting this?
%We have already had some confusion about what matrix norm is used.
%Some readers may be confused too.}
Given an $m\times n$ matrix $B$, the symbol $B^T$ stands for the transposed matrix,
and both $B$ and $B^T$  are identified with the corresponding linear operators acting between  $\R^m$ and $\R^n$ or their duals.
%(identified with the corresponding linear operator acting between the duals of $\R^m$ and $\R^n$).
%AK26/02/18
%We use the operator norm for matrices.

We denote by $F: \R^n \rightrightarrows \R^m$ a {\em
set-valued} mapping acting from  $\R^n$ to the subsets of $\R^m$.
If $F$ is a
function, that is, for each $x\in\R^n$ the set of values $F(x)$
consists of no more than one element, then we use a small letter $f$ and write $f:\R^n\to\R^m$.
The graph of a mapping $F$ is defined as $\gph F:= \{ (x,y)\in \R^n\times \R^m
\mid y \in F(x) \}$ and its domain is $\dom F:= \{x \in \R^n \mid F(x)
\neq \emptyset\}$.
The  inverse of a mapping $F$  is the mapping $y\mapsto  F^{-1}(y):=\{x\in\R^m\mid y\in F(x)\}.$
In this paper we consider mappings with {\em closed graph}.

The Lipschitz modulus of a function $f:\R^n\to\R^m$ around a point $\bx$ is defined by
$$
\lip (f;\bx):= \limsup_{\substack{x,x'\to\bx,\,x\neq x'}}
     \frac{\|f(x)-f(x')\|}{ \|x-x'\|}.
$$
Having $\lip(f;\bx)<l$ corresponds to having a neighborhood $U$ of $\bx$
such that $f$ is Lipschitz continuous on $U$ with Lipschitz constant $l$.
Conversely, if $f$ is Lipschitz continuous around $\bx$ with Lipschitz
constant $l$ then we have $\lip(f;\bx) \leq l$. If $f$ is not Lipschitz
continuous around $\bx$ then $\lip(f;\bx)=+\infty$.
%We denote by  ${\cal L}(\bx)$ the space of  functions $f:\R^n \to \R^m$ that are Lipschitz continuous  around $\bx$,

\if{
\AD{13/01/18. delete: We consider a mapping
$F:\R^n \rightrightarrows \R^m$
and a point $(\bx,\by)\in\gph F$.}}\fi
\if{\AD{
Insert all preparatory material including material
from the papers \cite{GfrOut16.2}, \cite{Gfr11}, definitions of $\overline{D}^*$,
$ {\rm Cr}_o F(\bx, \by)$, of semismoothness, etc...

{\bf Shall we do that? Helmut, Jiri, this is your  turf. }
}}\fi

\if{
 \AD{delete: Our aim  in this section is to establish
%AK11/02/18
\red{estimates for}
the radius $\rad F(\bx\for\by)$ of metric subregularity of $F$ at $\bx$ for $\by$.
As discussed in Section~\ref{intro}, we need to decide first on the families of legitimate additive perturbations of $F$ near $\bx$.
%AK11/02/18
\red{For simplicity we consider only perturbation functions $h:\R^n \to \R^m$ with $h(\bx) = 0$.}
Three
%Two
families of perturbations are considered below:}
}\fi

%Throughout the sequel, we will employ the following standard notions from variational analysis, see [RoWe98] and [Mor06]. \\
Given a closed set $A \subset\mathbb{R}^{n} $ and a point $\bar{x}\in A$, we define
\begin{enumerate}
\item
the {\em  tangent (Bouligand) cone} to $A$ at $\bar{x}$:
\[
T_{A}(\bar{x}):= \{u \in \mathbb{R}^{n}\mid \exists u_{i}\rightarrow u,\;t_{i}\searrow 0
\qdtx{such that}
\bar{x}+t_{i}u_{i}\in A,\; \forall i\in\N\};
\]
\item
the {\em Fr\'{e}chet normal cone} to $A$ at $\bar{x}$ as the (negative) \emph{polar} cone to $T_{A}(\bar{x})$:
\[
N_{A}(\bar{x}):=(T_{A}(\bar{x}))^{\circ} =\{x^*\in\R^n\mid \ang{x^*,u}\le0
\qdtx{for all}
u\in T_{A}(\bar{x})\};
\]
\item
the {\em limiting normal cone} to $A$ at $\bar{x}$:
\[
\overline{N}_{A}(\bar{x}):=\{x^*\in \mathbb{R}^{n}\mid \exists x_{i} \stackrel{A}{\rightarrow} \bar{x},\; x^*_{i}\rightarrow x^*
\qdtx{such that}
x^*_{i}\in{N}_{A}(x_{i}),\;\forall i\in\N\}.
\]
\end{enumerate}
If $\bx\notin A$, we use the convention that the three cones above are empty.

Given an extended-real-valued function $f:\R^n\to\R\cup\{+\infty\}$ and a point $\bx\in\dom f$, its \emph{limiting subdifferential} at $\bx$ can be defined by
\[\overline{\sd}f(\bx):=\{x^*\in\R^n\mid (x^*,-1)\in\overline{N}_{\epi f}(\bx,f(\bx))\},
\]
where $\epi f:=\{(x,\mu)\in\R^n\times\R\mid f(x)\le\mu\}$ is the \emph{epigraph} of $f$.
Given a function $f:\R^n\to\R^m$, Lipschitz continuous around a point $\bx\in\R^n$,
its \emph{Clarke generalized Jacobian} at $\bx$ is defined by
\[{\sd}_Cf(\bx) :=\co\left\{\lim_{k\to+\infty}\nabla f(x_k)\mid x_k\to\bx,\;f\mbox{ is differentiable at }x_k\right\},
\]
where $\co$ stands for the \emph{convex hull}.

Given a \SVM\ $F: \mathbb{R}^{n}\rightrightarrows \mathbb{R}^{m}$ and a point $(\bx,\by)\in \gph F$, the cones defined above give rise to the following {\em generalized derivatives}:

\begin{enumerate}
\item
the set-valued mapping $DF(\bx,\by):\mathbb{R}^{n} \rightrightarrows\mathbb{R}^m$, defined by
\[
DF(\bx,\by)(u):= \{v\in \mathbb{R}^m\mid (u,v)\in T_{\gph F}(\bx,\by)\},\quad u \in \mathbb{R}^{n},
\]
is called the {\em graphical derivative} of $F$ at $(\bx,\by)$;
\item
the set-valued mapping $D^{*}F(\bx,\by): \mathbb{R}^m\rightrightarrows\mathbb{R}^{n}$, defined by
\[
D^{*}F(\bx,\by)(v^{*}):=\{u^{*}\in \mathbb{R}^{n} \mid (u^{*},- v^{*})\in{N}_{\gph F}(\bx,\by)\},\quad v^{*}\in \mathbb{R}^m,
\]
is called the {\em Fr\'{e}chet coderivative} of $F$ at $(\bx,\by)$.
\item
the set-valued mapping $\overline{D}^{*}F(\bx,\by): \mathbb{R}^m\rightrightarrows\mathbb{R}^{n}$, defined by
\[
\overline{D}^{*}F(\bx,\by)(v^{*}):=\{u^{*}\in \mathbb{R}^{n} \mid (u^{*},- v^{*})\in \overline{N}_{\gph F}(\bx,\by)\},\quad v^{*}\in \mathbb{R}^m,
\]
is called the {\em limiting coderivative} of $F$ at $(\bx,\by)$.
\end{enumerate}

Recently, a finer, directionally dependent notion of a limiting normal cone has been introduced, cf. \cite{GinMor11,Gfr11,Gfr13}.
In addition to a set $A$ and a point $\bar{x}\in A$, one specifies also a direction $u\in \mathbb{R}^{n}$.
The cone
$$
\overline{N}_{A}(\bar{x};u):= \{x^*\in \mathbb{R}^{n}\mid\exists t_{i}\searrow 0, u_{i}\rightarrow u,x^*_{i}\rightarrow x^*
\mbox{ such that }
x^*_{i}\in N_{A}(\bar{x}+t_{i}u_{i}),\forall i\in\N\}
$$
is then called the {\em directional limiting normal cone} to $A$ at $\bar{x}$ in the direction $u$.

It is easy to see that $\overline{N}_{A}(\bar{x};u)= \emptyset$ when $u \not\in T_{A}(\bar{x})$ and
\begin{equation}\label{eq-100}
\overline{N}_{A}(\bar{x})= \bigcup_{\|u\|=1}\overline{N}_{A}(\bar{x};u)\cup N_{A}(\bar{x}).
\end{equation}
Relation (\ref{eq-100}) plays an important role in various conditions relaxing the standard criteria (sufficient conditions) for various Lipschitzian properties of set-valued mappings; see, e.g., \cite{GfrOut16.2,Gfr11}.

A set $A$ is called \emph{directionally regular} \cite{YeZho18} at $\bar{x}\in A$ in the direction $u$ if
$$
\overline{N}_{A}(\bar{x};u)= \{x^*\in \mathbb{R}^{n}\mid \forall t_{i}\searrow 0, \exists u_{i}\rightarrow u,x^*_{i}\rightarrow x^*
\mbox{ such that }
x^*_{i}\in N_{A}(\bar{x}+t_{i}u_{i}),\forall i\in\N\},
$$
%AK6/06/18
and simply \emph{directionally regular} at $\bar{x}$ if it is directionally regular at $\bar{x}$ in all directions.

Given a \SVM\ $F$, a point $(\bx,\by)\in\gph F$ and a pair of directions $(u,v)\in \mathbb{R}^{n}\times\mathbb{R}^m$, the set-valued mapping
$\overline{D}^{*}F((\bx,\by);(u,v)): \mathbb{R}^m\rightrightarrows \mathbb{R}^{n}$, defined by
\[
%%\begin{equation}\label{eq-8}
\overline{D}^{*}F((\bx,\by);(u,v))(v^{*}) :=\{u^{*}\in\mathbb{R}^{n}\mid(u^{*},-v^{*})\in \overline{N}_{\gph F} ((\bx,\by);(u,v))\},\quad v^{*}\in \mathbb{R}^m,
%%\end{equation}
\]
is called the {\em directional limiting coderivative of $F$} at $(\bx,\by)$ in the direction $(u,v)$.

With $F$ and $(\bx,\by)$ as above, the {\em limit set, critical for metric subregularity},
denoted by ${\rm Cr_0}F(\bx,\by)$, is the collection of all elements $(v,u^{*})\in\mathbb{R}^{m}\times\mathbb{R}^{n}$
such that there are sequences
$t_{i}\searrow0$, $(u_{i}),(u^{*}_{i})\subset\R^n$, $(v_{i}),(v^{*}_{i})\subset\mathbb{R}^{m}$ with $v_{i}\rightarrow v$,
$u^{*}_{i}\rightarrow u^{*}$,
\[
(-u^{*}_{i},v^{*}_{i})\in N_{\gph F}(\bar{u}+t_{i}u_{i}, \bar{v}+t_{i}v_{i})
\qdtx{and}
\|u\|=\|v^*\|_*=1.
\]
As proved in \cite[Theorem~3.2]{Gfr11}, the condition $(0,0)\notin{\rm Cr_0}F(\bx,\by)$ is sufficient for metric subregularity of $F$ at $\bar{u}$ for $\bar{v}$.

In our analysis we make use also of a generalization of the semismoothness property, introduced by Mifflin in \cite{Mif77}.
A function $f:\mathbb{R}^{n}\to\mathbb{R}^{m}$ is (weakly) {\em semismooth} at $\bar{x}$, provided it is Lipschitz continuous around $\bar{x}$ and the limit
\begin{equation}\label{eq-101}
\lim\{ V u^{\prime}\mid
V\in{\partial}_Cf(\bar{x}+t u^{\prime}),\;
u^{\prime}\rightarrow u,\; t \searrow 0\}
\end{equation}
exists for all $u \in \mathbb{R}^{n}$;
here
$\partial_Cf$ stands
%in (\ref{eq-101})
for the \emph{Clarke generalized Jacobian} of $f$.
It is easy to verify that this property implies directional differentiability of $f$ at $\bar{x}$ and limit (\ref{eq-101}) amounts to $f^{\prime}(\bar{x};u)$ (the \emph{Hadamard directional derivative} of $f$ at $\bar{x}$ in the direction $u$).

\subsection{Classes of perturbations and definitions of the radii}

As
%AK26/02/18
%mentioned
discussed
in Section~\ref{intro}, the radius of subregularity
%will depend
depends
on the choice of the class of functions that are used as perturbations.
We consider three such classes: Lipschitz
%AK7/04/18
{continuous, semismooth}
%Lipschitz
and $C^1$ functions.
%Specifically, we consider the following classes:}
%and linear (affine).
\begin{align*}
\mathcal{F}_{Lip}&:=\{h:\R^n \to \R^m\mid h \mbox{ is Lipschitz continuous around }\bx\},
\\
\mathcal{F}_{ss}&:=\{h:\R^n \to \R^m\mid h \mbox{ is %AK7/04/18
%Lipschitz continuous around}\bx,
%\\
%&\hspace{4.2cm}\mbox{
semismooth at }\bx\},
\\
\mathcal{F}_{C^1}&:=\{h:\R^n \to \R^m\mid h \mbox{ is } C^1 \mbox{ around }\bx\}.
\end{align*}

\if{
\AD{Semismoothness requires Lipschitz continuity around the point, right?

\red{AK26/02/18: According to Dontchev-Rockafellar, p.~411, the answer is yes.
%AK7/04/18
%Shall we say this explicitly and remove the Lipschitz continuity from the second definition above?
}}
}\fi

Without loss of generality,
%AK26/02/18
%in order  to simplify the notation, we restrict these classes to the case when
we will assume that perturbation functions $h$ in all three definitions satisfy
$h(\bx)=0$.
%AK26/02/18
%Unlike the cases of (strong) metric regularity and strong subregularity, as illustrated by
%AK14/04/18
%the example in Section~\ref{intro},
%{Example~\ref{E1}},
%linear perturbations do not seem appropriate when determining the radius of subregularity.
\AD{26/04/18.
I am not sure this is correct. Example~\ref{E1} is only for polyhedral mappings.}
\AK{16/05/18.
I think it is OK: if linear perturbations do not work in the polyhedral case, what can we expect in more general settings?
Besides, the phrase is sufficiently vague: `do not seem appropriate'...}

\if{\AK{14/01/18.
It would be great to allow for the third one:}
\magenta{\begin{align*}
\mathcal{F}_{Lin}&:=\{h:\R^n \to \R^m\mid h(x)=B(x-\bx)\; (x\in\R^n),\;
B\in\R^{m\times n}\}.
\end{align*}
\AK{
I have tried and failed.
This seems to be due to the specifics of metric subregularity compared to its stronger siblings.}
}\fi

%AK26/02/18
The corresponding radii are defined as follows:
\begin{align*}
\rad_{\mathcal{P}}F(\bx\for\by) :=\inf_{h\in\mathcal{F}_{\mathcal{P}}}\{\lip(h;\bx) &\mid F+h \mbox{ is not metrically subregular at $\bx$ for } \by\},
\if{\\
\rad_{Lip}F(\bx\for\by) :=\inf_{h\in\mathcal{F}_{Lip}}\{\lip(h;\bx) &\mid F+h \mbox{ is not metrically subregular at $\bx$ for } \by\},
\\
\rad_{ss}F(\bx\for\by) :=\inf_{h\in\mathcal{F}_{ss}}\{\lip(h;\bx) &\mid F+h \mbox{ is not metrically subregular at $\bx$ for } \by\},
\\
\rad_{C^1}F(\bx\for\by) :=\inf_{\substack{h\in\mathcal{F}_{C^1}}}\{\lip(h;\bx) &\mid F+h \mbox{ is not metrically subregular at $\bx$ for } \by\}.}\fi
%\\
%\rad_{Lin}F(\bx\for\by) &:=\inf_{\substack{h(x)=B(x-\bx)}}\{\|B\| \mid F+h \mbox{is not metrically subregular at $\bx$ for} \by\}.
\end{align*}
where $\mathcal{P}$ stands for $Lip$, $ss$ or $C^1$.
Note that, for every $h\in \mathcal{F}_{Lip}$,
%AK28/02/18
in view of \cite[Theorem~9.62]{RocWet98}
it holds
$$\lip(h;\bx)=\sup\{\norm{B}\mid B\in \partial_Ch(\bx)\},$$
where $\partial_Ch(x)$ stands for the Clarke generalized Jacobian of $h$ at $x$.

%30/03/18
%The above definition of the radii remains valid when $F$ is not subregular at $\bx$ for $\by$, in which case by convention the subregularity radius with respect to any class of perturbations equals 0: in this case one can take $h=0$ in the \RHS\ of the definition.
{
If $\rad_{\mathcal{P}}F(\bx\for\by)>0$, then $F$ is necessarily subregular at $\bx$ for $\by$ since $0\in\mathcal{F}_{\mathcal{P}}$ whenever $\mathcal{P}$ stands for any of the three classes considered in this paper.
In the degenerate case when $F$ is not subregular at $\bx$ for $\by$, the above definition of the radius automatically gives $\rad_{\mathcal{P}}F(\bx\for\by)=0$.
}
%AK11/02/18
%Observe that under the assumption $h(\bx) = 0$ the radius $\rad_{Lip}F(\bx\for\by)$ coincides with $\rad F(\bx\for\by)$ defined in \eqref{subrad1}.
%When it does not cause confusion we will keep writing $\rad F(\bx\for\by)$ (without the subscript) in the Lipschitz case.

We obviously have
$
%\mathcal{F}_{Lin}\subset
\mathcal{F}_{C^1}\subset \mathcal{F}_{ss}\subset\mathcal{F}_{Lip}$ and
\begin{gather}\label{8}
\rad_{Lip}F(\bx\for\by)\le\rad_{ss}F(\bx\for\by) \le\rad_{C^1}F(\bx\for\by).
%\le\rad_{Lin}F(\bx\for\by).
\end{gather}

%AK25/05/18
\subsection{Primal-dual derivative and regularity constants}

%AK12/02/18

Given $(\bx,\by)\in\gph F$, we define
%a \red{\underline{mixed primal-dual derivative-like object}}
{the \emph{primal-dual derivative}
${\widehat{D}F(\bx,\by)}:\R^n\times(\R^m)^* \allowbreak \rightrightarrows(\R^n)^*\times\R^m$
of $F$ at $(\bx,\by)$}
as follows: for all $(u,v^*)\in\R^n\times(\R^m)^*$,
\if{\AD{26/04/18.
I do not like the name ``derivative-like object". Does not sound enough English to me ---
why not just derivative?}
\AK{16/05/18.
Neither do I, and not only in terms of English.
That is why it is still in red.
I want to find a good name, which other people will accept.
If we do not invent one quickly, we will have to adopt `primal-dual derivative'.
If this object proves useful, other people will come up with a better name for it, and everybody will cite them.}
\HG{23/05/18.
Maybe the following geometric considerations are useful to find a better name: As Alex remarked in some earlier version, the mapping $\widehat{D}F(\bx,\by)$ is some combination of the graphical derivative and the limiting coderivative. In fact, as stated in Proposition 3  we always have
$\widehat{D}F(\bx,\by)(u,v^*)\subset  \overline{D}^{*}F(\bx,\by)(v^*)\times DF(\bx,\by)(u)$
and thus  tangents (related to the graphical derivative) are linked with limiting normals (related to the coderivative) to the graph of $F$ in a suitable way, i.e. we do not consider the graphical derivative and coderivative as independent objects. Hence a possible name could be ''combined graphical derivative/coderivative'' but this does not sound good for me.}
}\fi
\begin{gather}\label{pdd}
\widehat{D}F(\bx,\by)(u,v^*):=  \bigl\{(u^*,v)\in(\R^n)^*\times\R^m\mid
(u^*,-v^*)\in\overline{N}_{\gph F} ((\bx,\by);(u,v))\bigr\}.
\end{gather}
In other words,
\begin{gather*}%\label{pdd}
\widehat{D}F(\bx,\by)(u,v^*)=  \bigl\{(u^*,v)\in(\R^n)^*\times\R^m\mid
u^*\in\overline{D}^{*}F((\bx,\by);(u,v))(v^{*})\bigr\}.
\end{gather*}
\if{\AK{3/03/18.
I have realized that with the above definition I accidentally rediscovered the wheel.
My apologies to Helmut.
The `mixed primal-dual derivative' as well as the set $\widehat{\mathfrak{D}}F(\bx,\by)$ below are close relatives of the `critical set'.
My arguments for keeping them are not strong, especially the second one:

1) The `derivative' looks natural, just like the coderivative and tangential/graphical derivative with which it has a few things in common.
I intentionally put the signs as in the definition of the coderivative to simplify the comparison.
They differ from those in the definition of the critical set.
In this paper the signs do not matter as we use the norms of the vectors.

2) In the definition of $\widehat{\mathfrak{D}}F(\bx,\by)$ below I have swapped the output variables compared to the definition of the critical set just to position them in the alphabetical order: $u^*,v$.

Helmut, the decision is yours: thumbs up or down?
I am ready to remove all this derivative-like stuff and stick to the critical set (and its modifications) if you say so.

Could the definitions \eqref{pdd}, \eqref{DF1} or \eqref{DF2} be found elsewhere?
}
\HG{4/03/18. Unfortunately I did not see in 2010 that the critical limit set is related to the primal-dual derivative. I recognized this only one year later but  instead I used then the directional limiting coderivative  allowing a finer analysis. However, for estimating the radius of metric subregularity the primal-dual derivative is well suited and simplifies the notation.
}}\fi
%AK2/03/18
%AK25/05/18
{The next proposition, which follows directly from the definitions, shows that the mapping $\widehat{D}F(\bx,\by)$ combines features of the graphical derivative and the limiting coderivative:
tangents (related to the graphical derivative) are linked with limiting normals (related to the coderivative) to the graph of $F$ in a suitable way.}

\begin{proposition}\label{P1}
{$\widehat{D}F(\bx,\by)(u,v^*)\subset  \overline{D}^{*}F(\bx,\by)(v^*)\times DF(\bx,\by)(u)$ for all $(u,v^*)\in\R^n\times(\R^m)^*$.}
\end{proposition}
\AK{17/02/18.
Please comment on the above definition and the corresponding definitions of $\rg$, $\rgd$ and $\rgo$ below as well as the name(s), notations and signs.
Could these constructions be useful elsewhere?}

Using \eqref{pdd} we define
%AK14/04/18
%the sets:
{two image sets under $\widehat{D}F(\bx,\by)$}:
\begin{align}\label{DF}
\widehat{\mathfrak{D}}F(\bx,\by):=
&\big\{(u^*,v)\in\widehat{D}F(\bx,\by)(u,v^*)\mid
\|u\|=\|v^*\|_*=1\big\},
\\\label{DF1}
\widehat{\mathfrak{D}}^\circ F(\bx,\by):=
&\big\{(u^*,v)\in\widehat{D}F(\bx,\by)(u,v^*)\mid
\|u\|=\|v^*\|_*=1,\;
{u^\ast}^Tu={v^*}^Tv\big\}.
\end{align}
Observe that the set \eqref{DF} is a small modification of the limit set ${\rm Cr_0}F(\bx,\by)$
%critical for metric subregularity
\cite{Gfr11}:
${(u^*,v)\in\widehat{\mathfrak{D}}F(\bx,\by)}$ if and only if $(v,-u^*)\in{\rm Cr_0}F(\bx,\by)$.
\sloppy
\if{\AD{26/04/18.
This should go with a detailed definition and statement of the characterization at the end of Subsection 2.1.}
}\fi

%AK25/05/18
{
\begin{proposition}\label{P4}
The image set \eqref{DF1} admits an equivalent representation involving an ${m\times n}$ matrix:
\begin{multline}\label{DF2}
\widehat{\mathfrak{D}}^\circ F(\bx,\by)=
\big\{(u^*,v)\in
\widehat{D}F(\bx,\by)(u,v^*)\mid
\|u\|=\|v^*\|_*=1,
\\
B^Tv^*=u^\ast,\; Bu=v,\;
B\in L(\R^n,\R^m)\big\}.
\end{multline}
\end{proposition}
}
\begin{proof}
{Let $u,u^*\in\R^n$, $v,v^*\in\R^m$ and $\|u\|=\|v^*\|_*=1$.
We need to check the equivalence of the condition ${u^\ast}^Tu={v^*}^Tv$ to the pair of conditions $Bu=v$ and $B^Tv^*=u^*$ for some $m\times n$ matrix $B$.}

{Suppose that ${u^\ast}^Tu={v^*}^Tv.$
Choose vectors $z^*\in\R^n$ and $w\in\R^m$ such that $\|w\|=\|z^*\|_*={z^*}^Tu={v^*}^Tw=1$, and
set
\begin{gather}\label{B}
B:=v{z^*}^T+w{u^*}^T-({u^*}^Tu)w{z^*}^T.
\end{gather}
Then $Bu=v+({u^*}^Tu)w-({u^*}^Tu)w=v$ and $B^Tv^*=(v^Tv^*){z^*}+{u^*}-({u^*}^Tu){z^*}=u^*$.}

{Conversely, suppose that $Bu=v$ and $B^Tv^*=u^*$ for some $m\times n$ matrix $B$.
Then
${u^\ast}^Tu=u^Tu^\ast=u^TB^Tv^*={v^*}^TBu={v^*}^Tv.$}
\qed\end{proof}

{
\begin{remark}
The above proof of Proposition~\ref{P4} is constructive.
In the first part, it not only establishes the existence of a matrix $B$ with required properties; it provides the formula \eqref{B} for constructing such a matrix.
\end{remark}
}

The following quantities
%AK2/03/18
%seem to be quite
are
instrumental in deriving bounds for the radius of metric subregularity:
\begin{align}\label{rg}
\rg:= &\inf \bigl\{\max\{\norm{u^*}_*,\norm{v}\} \mid
(u^*,v) \in\widehat{\mathfrak{D}}F(\bx,\by)\bigr\},
%AK14/04/18
%\\
%\rgo:= &\inf \bigl\{\norm{B}\mid
%B\in\widehat{\mathfrak{D}}^\circ F(\bx,\by)\bigr\}.
\\\notag
\rgo:= &\inf \bigl\{\norm{B}\mid
{B\in L(\R^n,\R^m),}
\\\label{rgo}
&\hspace{18mm}
{(B^Tv^*,Bu)\in\widehat{\mathfrak{D}}^\circ F(\bx,\by),\;
\|u\|=\|v^*\|_*=1}
\bigr\}.
\end{align}
%AK30/03/18
The next two modifications of \eqref{rg} and \eqref{rgo} can also be useful:
\begin{align}\label{rgh}
\rgh:= &\inf \bigl\{\norm{u^*}_*+\norm{v} \mid
(u^*,v) \in\widehat{\mathfrak{D}}F(\bx,\by)\bigr\},
\\\label{rgd}
\rgd:= &\inf \bigl\{\max\{\norm{u^*}_*,\norm{v}\} \mid
(u^*,v) \in\widehat{\mathfrak{D}}^\circ F(\bx,\by) \bigr\}.
\end{align}
%AK25/05/18
{They provide, respectively, an upper bound for \eqref{rg} and a lower bound for \eqref{rgo}.
This explains their notations.
Note that \eqref{rgd} is also an upper bound for \eqref{rg}.}
\if{\HG{24/05/18.
I suggest a change of the notation: $\rg$ and $\rgh$ are lower and upper bounds for $\rad_{Lip}F(\bx\for\by)$. Why do we not call it $\underline{{\rm rg}}[F](\bx,\by)$ and
$\overline{{\rm rg}}[F](\bx,\by)$? Similarly, we could use $\underline{{\rm rg}}^\circ[F](\bx,\by)$ instead of $\rgd$, i.e. we use an underline for a lower bound and an overline for an upper bound}
}\fi

\begin{proposition}\label{P3}
\begin{enumerate}
\item
$\rg\leq\rgd\leq\rgo$;
\item
%AK30/03/18
$\rg\leq\rgh
%AK25/05/18
{\leq2\rg}$;
\item
$\rg\ge\inf \{\norm{u^*}_*\mid  u^\ast\in\overline{D}^*F(\bx,\by)(v^*),\ \norm{v^*}_*=1\}$;
\item
$\rg\ge\inf \{\norm{v}\mid v\in DF(\bx,\by)(u),\ \norm{u}=1\}$.
\end{enumerate}
\end{proposition}

\AK{30/03/18.
Any ideas about the relationship between $\rgo$ and $\rgh$?
In the example in Section~\ref{S4}, we have $\rgo<\rgh$.}
\AK{4/03/18.
What do constants \eqref{rg}, \eqref{rgo}, \eqref{rgh} and \eqref{rgd} say about $F$ apart from providing estimates for the subregularity radii?}
\if{\HG{24/05/18.
I suggest extending Proposition \ref{P3}(ii) to $\rg\leq\rgh\leq 2\rg$.}
}\fi
\if{\HG{4/03/18.
We should mention somewhere that the quantities on the right hand side of the inequalities in (ii) and (iii) are exactly the radii for metric regularity and strong metric subregularity
}}\fi

\begin{proof}
(i) In view of the definitions \eqref{rg} and \eqref{rgd}, the first inequality is immediate from comparing \eqref{DF} and \eqref{DF1}.
%AK25/05/18
%Let $B^Tv^*=u^\ast$ and $Bu=v$.
%Then
%${u^\ast}^Tu=u^Tu^\ast=u^TB^Tv^*={v^*}^TBu={v^*}^Tv.$
\if{\AK{14/04/18.
Could this assertion be `reversed': given four vectors, two of which are nonzero, satisfying the compatibility condition, could a matrix $B$ be found?
It is about solvability of a system of $n+m+1$ linear equation with $nm$ variables.
Each of the three subsystems involved has rank$=1$.
The number of equation exceeds the number of variables only when $n=1$ or $m=1$.
It is easy to check that in this case the system is nevertheless solvable.
}
\HG{24/05/18.
Yes, this assertion can be reversed; just take $B=\frac {vu^T}{\norm{u}_2^2}+\frac{v^*{u^*}^T}{\norm{v^*}_2^2}-v^*u^T\frac{{u^*}^Tu}{\norm{u}_2^2\norm{v^*}_2^2}$. As a consequence we have $\widehat{\mathfrak{D}}^\diamond F(\bx,\by)=\widehat{\mathfrak{D}}^\circ F(\bx,\by)$ and the definition of one of these sets can be omitted. I suggest to omit the definition of  $\widehat{\mathfrak{D}}^\circ F(\bx,\by)$ and then to include a lemma which states $\widehat{\mathfrak{D}}^\diamond F(\bx,\by)=\big\{(u^*,v)\in
{\widehat{D}F(\bx,\by)(u,v^*)\mid
\|u\|=\|v^*\|_*=1,}
{B^Tv^*=u^\ast,\; Bu=v,\;
B\in L(\R^n,\R^m)}
\big\}.$}
}\fi
%If, additionally, $\norm{v^\ast}_*=\norm{u}=1$, then
%$\norm{u^*}_*=\norm{B^Tv^*}_*\leq \norm{B^T}\norm{v^*}_*=\norm{B}$ and $\norm{v}=\norm{Bu}\leq \norm{B}\norm{u}=\norm{B}$.
In view of the definitions \eqref{rgo} and \eqref{rgd}, the second inequality follows from
Proposition~\ref{P4}.

%AK25/05/18
%Condition (ii) is
{The inequalities in (ii) are}
immediate from comparing the definitions \eqref{rg} and \eqref{rgh}.

Inequalities (iii) and (iv) are consequences of the definitions \eqref{DF} and \eqref{DF1}, and Proposition~\ref{P1}.
\qed\end{proof}

\AK{3/03/18.
Are the inequalities in Proposition~\ref{P3}(i) strict?
Counterexamples?}
%AK 7/04/18
{The quantities in the right-hand sides of the inequalities in parts (iii) and (iv) of Proposition~\ref{P3} equal to the reciprocals of the moduli of the metric regularity and strong metric subregularity, respectively; cf. \cite[Theorems~4C.2 and 4E.1]{DonRoc14}, and in view of \cite[Theorems~6A.7 and 6A.9]{DonRoc14},
are exactly the radii of the corresponding properties.
Thus, the value of $\rg$ is an upper bound for both these radii.
}

\section{The radius theorem}\label{main}

In this section we present the main results of this paper.
\AD{26/04/18.
Why do we need subsections in Section 3? Just  add to the end of the section discussion of the norms.}
\AK{17/05/18.
No objections.
I keep subsections temporarily to simplify observing the overall structure through the table of contents, and also for the unlikely case that Subsection 3.2 grows.}

%\subsection{Main theorem}
We start with a lemma which is a
%AK 26/02/18
%direct
consequence of \cite[Theorem~10.41
and Exercise~10.43]{RocWet98}.

%AK 7/02/18
%New lemma and updated theorem thanks to Jiri
\begin{lemma}\label{L2}
Consider a mapping
$F:\R^n \rightrightarrows \R^m$
%AK28/02/18
with closed graph,
a function $h:\R^n \to \R^m$ and a point $(u,v) \in \gph (F+h)$ such that $h$ is
Lipschitz continuous around $u$.
Then
\begin{gather}\label{L2-1}
\overline{D}^*(h+F)(u,v)(v^*)\subset\overline{D}^*h(u)(v^*) +\overline{D}^*F(u, v - h(u))(v^*)
\quad\mbox{for all}\;\;
v^*\in\R^m.
\end{gather}
As a consequence,
\begin{multline*}%\label{L2-2}
\overline{N}_{\gph(h+F)}(u,v)
%\label{EqInclNormalCone}
\subset \{(u_h^*+u_F^*,v^*)\mid
\\
(u_h^*,v^*)\in \overline{N}_{\gph h}(u,h(u)),\;
(u_F^*,v^*)\in \overline{N}_{\gph F}(u,v-h(u))\}.
%\\
%\subset \{(u^*,v^*)\mid (u^*+B^Tv^*,v^*)\in \overline{N}_{\gph F}(u,v-h(u)),\;B\in\sd_Ch(u)\}.
\end{multline*}
If, additionally, $h$ is strictly differentiable at $u$, then
\begin{gather*}%\label{L2-3}
\overline{N}_{\gph(h+F)}(u,v)= \{(u^*,v^*)\in\R^n\times\R^m\mid
(u^*+\nabla h(u)^Tv^*,v^*)\in \overline{N}_{\gph F}(u,v-h(u))\}.
\end{gather*}
\end{lemma}

\if{\AK{7/02/18: Could a version of \cite[Proposition 2.8]{Mor94.2} be found in \cite{RocWet98}, \cite{DonRoc14} or the book by Mordukhovich for a reference?
The result is well known.}}\fi

Our main result given next provides lower and upper bounds for the radius of metric subregularity for Lipschitzian perturbations and the exact radius
%AK26/02/18
formula
for
the
other classes of perturbations.
\AD{26/04/18.
In the Intro we say: ``In this paper we consider mappings with {\em closed graph}." But then why we assume closed graph in Theorem 7 and Corollaries 8 and 9?}
\AK{17/05/18.
I hesitate removing this assumption from the theorem and corollaries.
In the past, I always tried to avoid this type of repetitions.
However, on several occasions this led to some of my results cited wrongly (at least once by myself).
There was even an instance (or two) when other people were `correcting' my results by adding `missing' assumptions.
Now I do such repetitions all the time.}

\begin{theorem} \label{ThRadLip}
Consider a mapping $F:\R^n \rightrightarrows \R^m$
%AK1/03/18
with closed graph
and a point ${(\bx,\by)
%AK26/02/18
\in\gph F}.
$
%at which $F$ is metrically subregular.
Then
\begin{gather}\label{EqRadLip}
%\rad_{Lip}F(\bx\for\by)\geq \rg,
\rg\leq\rad_{Lip}F(\bx\for\by)\leq
%AK30/03/18
%AK14/04/18
\rgh,
%{\min\{\rgo,\rgh\}},
\\\label{1}
\rad_{ss}F(\bx\for\by)=\rad_{C^1}F(\bx\for\by)=\rgo.
\end{gather}
\end{theorem}

\if{\HG{10/2/2018

I think that also an upper bound $\rad_{Lip}F(\bx\for\by)\leq c\,\rg$ with some constant $c$ is valid. However, the proof is very complicated and  lengthy and so I decided not to work it out in full detail. The problem is that the required perturbation $h$ looks very ugly (e.g., it cannot be semismooth). Maybe somebody has a good idea to prove this upper bound.}

\AK{27/02/18.
It would be great to have such an estimate.
At the moment, I don't see how to get it.
$c$ must be independent of $F$.}}\fi

\if{\AK{14/01/18.
%It is not absolutely trivial to deduce this from \cite[Theorem 10.41]{RocWet98}.
%The statement could be added as a separate proposition.
Could the usage of the Clarke generalized Jacobian be avoided?

It would be good to have an analogue of this result in terms of Fr\'echet constructions.}}\fi

%Could the assumptions of closed graph and semismoothness be eliminated?}

\begin{proof}
\underline{Step 1: $\rg\leq\rad_{Lip}F(\bx\for\by)$}.
Let $h\in\mathcal{F}_{Lip}$
be  such that $h+F$ is not metrically subregular at $\bx$ for $\by$.
From \cite[Theorem 3.2]{Gfr11} we obtain that $(0,0) \in  {\rm Cr}_0 (h+F)(\bx,\by)$.
This implies that there exist sequences $t_k$, $u_k$, $v_k$, $u^*_k$, $v^*_k$ such that \begin{gather}\label{ThRadLipP1}
t_k\downto0,\;v_k\to0,\;u_k^*\to 0,\;
\|u_k\|=\|v_k^*\|_* = 1\;(k=1,2,\dots),
\\\notag
(-u_k^*, v_k^*)\in {N}_{\gph (h+F)}(\bx+t_ku_k, \by+t_kv_k)\;(k=1,2,\dots).
\end{gather}
By
%inclusion \eqref{EqInclNormalCone}
%AK11/02/18
Lemma~\ref{L2},
%(the first inclusion in \eqref{L2-2})}
there are elements $u_{h,k}^*$ such that
\begin{gather}\label{ThRadLipP2}
(u_{h,k}^*,v_k^\ast)\in \overline{N}_{\gph h}(\bx+t_ku_k,
h(\bx+t_ku_k)),
\\\label{ThRadLipP3}
(-u_k^*-u_{h,k}^*, v_k^*) \in \overline{N}_{\gph F}(\bx+t_ku_k, \by+t_k
\tilde v_k),
\end{gather}
%AK11/02/18
%Next consider any
where
\begin{gather}\label{ThRadLipP4}
\tilde v_k:=v_k - h(\bx+t_ku_k)/t_k.
\end{gather}
Let
$\gamma>\lip(h;\bx)$.
Then,
for all $k$ sufficiently large, in view of \cite[Proposition 9.24(b)]{RocWet98},
\begin{gather*}
\norm{u_{h,k}^*}_*\leq \gamma\norm{v_k^*}_*=\gamma,
\quad\mbox{and}
\\
\norm{\tilde v_k}
=\norm{v_k-(h(\bx+t_ku_k)-h(\bx))/t_k} \leq\norm{v_k}+\gamma\norm{u_k}=\norm{v_k}+\gamma.
\end{gather*}
Without loss of generality, we can assume that
\begin{gather}\label{ThRadLipP5}
u_k\to u,\;v^*_k\to v^*,\;
u_{h,k}^*\to u_h^*,\;
\tilde v_k
\to\tilde v,\;
\|u\|=\|v^*\|_*=1.
\end{gather}
We conclude that $\norm{\tilde v}\leq \gamma$, $\norm{u_h^*}_*\leq \gamma$ and
$(-u_h^*,v^*)\in \overline{N}_{\gph F}((\bx,\by),(u,\tilde v))$,
%AK15/02/18
i.e. ${(-u_h^*,\tilde v)\in\widehat{D}F(\bx,\by)(u,-v^*)}$.
Thus, $\rg\leq\max\{\norm{u_h^*}_*,\norm{\tilde v}\}\leq \gamma$.
%AK16/02/18
%and, since $\gamma$ can be chosen arbitrarily close to $\lip(h;\bx)$, the bound $\rg\le\lip(h;\bx)$ follows.
%The inequality \eqref{EqRadLip} is now a direct consequence of the definition of $\rad_{Lip}F(\bx\for\by)$.
Taking infimum in the last inequality over all $\gamma>\lip(h;\bx)$ and then over all  $h\in\mathcal{F}_{Lip}$
such that $h+F$ is not metrically subregular at $\bx$ for $\by$, we arrive at
%\blue{the first inequality in }\eqref{EqRadLip}.
$\rg\leq\rad_{Lip}F(\bx\for\by)$.

%Next we show the second inequality in \eqref{EqRadLip}.
\underline{Step 2: $\rad_{Lip}F(\bx\for\by)\leq\rgh$}.
%AK12/03/18
%We are going to construct a special Lipschitz continuous perturbation function $h$ of the \SVM\ $F$.
To show this inequality, we construct a special Lipschitz continuous perturbation $h$.
Given a strictly decreasing sequence $\tau=(\tau_k)$ of positive real numbers converging to $0$, set for every $k=1,2,\ldots$,
$$a_{k+1}:=\tau_{k+1}+\frac{\tau_k-\tau_{k+1}}{2(k+1)}, \quad b_k:=\tau_{k}-\frac{\tau_k-\tau_{k+1}}{2(k+1)}.$$
%AK10/03/18
Then $\tau_{k+1}<a_{k+1}<b_k<\tau_k$.
Define a function $\chi_\tau:\R\to\R$ recursively as follows:
\[\chi_\tau(t):=
\begin{cases}
0&\mbox{if } t=0,
\\
\chi_\tau(a_{k+1})-a_{k+1}+t&\mbox{if } a_{k+1}<t\le b_k,
\\
\chi_\tau(b_k)&\mbox{if } b_k<t\le a_k,
\\
-\chi_\tau(-t)&\mbox{if } t<0.
\end{cases}
\]
Thus, $\chi_\tau$ is linear on every interval $[a_{k+1},b_k]$ and $[-b_k,-a_{k+1}]$ with slope $1$ and constant on every interval $[b_k,a_k]$ and $[-a_k,-b_k]$. In particular, $\chi_\tau$ is Lipschitz continuous on $\R$ with modulus $1$ and continuously differentiable at $\tau_k$ with the derivative equal $0$.
%AK9/03/18
Moreover, for all $t\in(b_k,a_k)$, we have
\begin{gather*}
\chi_\tau(t)=\sum_{j=k}^\infty(b_j-a_{j+1})
=\sum_{j=k}^\infty (\tau_{j}-\tau_{j+1})\left(1-\frac{1}{j+1}\right)
= \tau_k-\sum_{j=k}^\infty\frac{\tau_j-\tau_{j+1}}{j+1},
\end{gather*}
and consequently,
\begin{align*}
\tau_k>\chi_\tau(\tau_k) >\tau_k-\frac{1}{k+1}\sum_{j=k}^\infty(\tau_j-\tau_{j+1})
=\tau_k\left(1-\frac 1{k+1}\right),
\end{align*}
showing
\[\lim_{k\to+\infty}\frac{\chi_\tau(\tau_k)}{\tau_k} =\lim_{k\to+\infty}\frac{\chi_\tau(-\tau_k)}{-\tau_k}=1.\]

Next, consider $(u^*,v) \in\widehat{D}F(\bx,\by)(u,v^*)$ with $\|u\|=\|v^*\|_*=1$ and choose elements $\hat u^\ast\in(\R^n)^\ast$ and $\hat v\in\R^m$ with $\norm{\hat u^\ast}_*=\norm{
%AK9/03/18
%v
\hat v
}=1$ such that
\[{\hat u{}^\ast}^Tu=\norm{u}=1,\quad {v^\ast}^T
%AK9/03/18
%v
\hat v
=\norm{v^\ast}_*=1.\]
By the definition of $\widehat{D}F(\bx,\by)$, there exist sequences $t_{k}\searrow 0$, $u_{k}\rightarrow u$, $v_{k}\rightarrow v$, $u^{*}_{k}\rightarrow u^{*}$ and $v^{*}_{k}\rightarrow v^{*}$ such that
\[(u_k^*,-v_k^*)\in N_{\gph F}(\bx+t_ku_k,\by+t_kv_k).\]
By passing to a subsequence if necessary, we can assume that the sequence $\tau_u:=(t_k{\hat u{}^\ast}^Tu_k)$ is strictly decreasing.
If ${u^\ast}^Tu\ne0$, we can also assume that the sequence $\tau_{u^\ast}:=(t_k\vert {u^\ast}^Tu_k\vert)$ is strictly decreasing; in this case we set $\zeta(x):={u^\ast}^Tx-\chi_{\tau_{u^\ast}}({u^\ast}^Tx)$, and
observe that
\[\lim_{k\to+\infty}\frac{\zeta(t_ku_k)}{t_k} =({u^\ast}^Tu)\lim_{k\to+\infty}\frac{\zeta(t_ku_k)}{t_k {u^\ast}^Tu_k} =({u^\ast}^Tu)\left(1-\lim_{k\to+\infty} \frac{\chi_{\tau_{u^\ast}}(t_k{u^\ast}^Tu_k)} {t_k{u^\ast}^Tu_k}\right)=0.\]
When ${u^\ast}^Tu=0$, we set $\zeta(x):={u^\ast}^Tx$ and
observe that $\zeta(t_ku_k)/t_k={u^\ast}^Tu_k\to0$ as ${k\to+\infty}$.
In both cases, $\zeta$ is Lipschitz continuous on $\R^n$ with modulus $\|u^*\|_*$ and continuously differentiable at $t_ku_k$ with the derivative $\nabla \zeta(t_ku_k)=u^\ast$.
Next, consider the mapping $h:\R^n\to\R^n$ given by
\[h(x):=
\chi_{\tau_u}({\hat u{}^\ast}^T(x-\bx))v+
\zeta(x-\bx)\hat v,\quad x\in\R^n.\]
We have
\[\lim_{k\to+\infty}\frac{h(\bx+t_ku_k)}{t_k} =\lim_{k\to+\infty}\left(\frac{\chi_{\tau_u}(t_k{\hat u{}^\ast}^Tu_k)}{t_k}v+\frac{\zeta(t_ku_k)}{t_k}\hat v\right) =v.\]
Further, $h$ is continuously differentiable at $\bx+t_ku_k$ with the derivative $\nabla h(\bx+t_ku_k)=\hat v{u^\ast}^T$, implying $\nabla h(\bx+t_ku_k)^Tv_k^*=u^\ast(\hat v^Tv_k^\ast)\to u^\ast$. By virtue of Lemma~\ref{L2}, we obtain
 \[(u_k^\ast -u^\ast(\hat v^Tv_k^\ast), -v_k^\ast)\in N_{\gph (F-h)}(\bx+t_ku_k, \by +t_kv_k-h(\bx+t_ku_k)).\]
 Since $u_k^\ast -u^\ast(\hat v^Tv_k^\ast)\to 0$ and $v_k-h(\bx+t_ku_k)/t_k\to0$ as $k\to 0$, we obtain
 that $(0,0)\in {\rm Cr}_{0}(F-h)(\bx,\by)$.
By \cite[Theorem 3.2(2)]{Gfr11}, we can now find a $C^{1}$ perturbation $\tilde{h}$ with $\tilde{h}(\bar{x})=0$ and $\| \nabla \tilde{h}(\bar{x})\| = 0$ such that $F-h+\tilde{h}$ is not metrically subregular at $(\bar{x},\bar{y})$.
We now want to estimate $\lip(h;\bx)$.
Taking any $x_1,x_2\in\R^n$, we have
\begin{align*}
\norm{h(x_1)-h(x_2)}&\leq (\norm{v}\norm{\hat u^\ast}_*+\norm{\hat v}\norm{u^\ast}_*)\norm{x_1-x_2}=(\norm{v}+\norm{u^\ast}_*)\norm{x_1-x_2}.
\end{align*}
Hence, $\lip(h-\tilde h;\bx)=\lip(h;\bx)\leq \norm{v}+\norm{u^\ast}_*$ and,
since $(h-\tilde{h})(\bar{x})=0$, we conclude that $\rad_{Lip}F(\bx\for\by)\leq \norm{v}+\norm{u^\ast}_*$. The inequality
%the second inequality in \eqref{EqRadLip}
$\rad_{Lip}F(\bx\for\by)\leq\rgh$
follows.
This completes the proof of \eqref{EqRadLip}.

\underline{Step 3: $\rgo\leq\rad_{ss}F(\bx\for\by)$}.
Let $h\in\mathcal{F}_{ss}$
be  such that $h+F$ is not metrically subregular at $\bx$ for $\by$.
Then $h\in\mathcal{F}_{Lip}$ and, as shown above, there exist sequences $t_k$, $u_k$, $v_k$, $\tilde v_k$, $u^*_k$, $u^*_{h,k}$, $v^*_k$ and vectors $u$, $\tilde v$, $v^*$, $u_h^*$ such that conditions \eqref{ThRadLipP1}, \eqref{ThRadLipP2}, \eqref{ThRadLipP3}, \eqref{ThRadLipP4} and \eqref{ThRadLipP5} hold true.
Thanks to the Lipschitz continuity of $h$, it follows from \eqref{ThRadLipP2} that $-u_{h,k}^*\in\overline{\sd}\ang{v_k^*,h}(\bx+t_ku_k)$ (cf. e.g. \cite[Proposition~9.24]{RocWet98}), and consequently, $
%AK10/03/18
%-
u_{h,k}^*
\in B_k^Tv_k^*$
%for some $B_k
where
$
-B_k
\in\sd_Ch(\bx+t_ku_k)$.
From the Lipschitz continuity of $h$,
the sequence of matrices $B_k$ is bounded.
Without loss of generality, we can assume that $B_k\to B\in
-
\partial_Ch(\bx)$.
Note that $\|B\|\le\lip(h;\bx)$.
Thus, $
%-
u_{h}^*
=B^Tv^*$.
Since $h$ is semismooth, it is directionally differentiable in all directions, and, in view of \eqref{ThRadLipP4} $\tilde v_k\to-h'(\bx;u)=
%-
B
u$.
It now follows from \eqref{ThRadLipP3} that
$(
-
B^Tv^*,v^*)
\in\overline{N}_{\gph F}((\bx,\by);(u,
%-B
B
u))$,
i.e. $(B^T
%v^*,
(-v^*),
%-B
B
u)\in\widehat{D}F(\bx,\by)(u,-v^*)$.
%Hence,
Since $\norm{u}=\norm{-v^*}_*=1$, we have
$\rgo\le\|B\|\le\lip(h;\bx)$.
Taking infimum in the last inequality over all $h\in\mathcal{F}_{ss}$
such that $h+F$ is not metrically subregular at $\bx$ for $\by$, we arrive at $\rgo\leq\rad_{ss}F(\bx\for\by)$.

\underline{Step 4: $\rad_{C^1}F(\bx\for\by)\le \rgo$}.
Suppose $\rgo<+\infty$; otherwise there is nothing to prove.
Let $\eps>0$.
It follows from
the definition of $\rgo$
that there exists a matrix $B\in \R^{m\times n}$ with $\|B\|<\rgo+\eps$ and  vectors $u\in\mathbb{R}^{n}$ and $v^{*}\in\mathbb{R}^{m}$ with ${\|u\|=\|v^*\|_*=1}$ such that
$(B^Tv^*,-v^*)\in\overline{N}_{\gph F}((\bx,\by);(u,Bu))$.
Hence, there exist sequences
$t_{k}\searrow 0$, $u_{k}\rightarrow u$, $v_{k}\rightarrow Bu$, $u^{*}_{k}\rightarrow B^{T}v^{*}$ and $v^{*}_{k}\rightarrow v^{*}$
such that
\begin{gather}\label{6}
(u^{*}_{k},-v^{*}_{k})\in N_{{\rm\gph} F}(\bar{x}+t_{k}u_{k}, \bar{y}+t_{k}v_{k})
\end{gather}
for all $k=1,2,\ldots$.
Set $h(x):=
%AK3/03/18
-
B(x-\bx)$ $(x\in\R^n)$.
Obviously, $h(\bx)=0$, $h$ is $C^{1}$ and $\nabla h(x)=
%AK3/03/18
-
B$ for any $x\in\R^n$.
Invoking
Lemma~\ref{L2},
%AK3/03/18
with $h+F$ and $-h$ in place of $F$ and $h$, respectively,
we obtain from \eqref{6} that
\begin{gather*}%\label{6}
(\hat u^{*}_{k},-v^{*}_{k})\in N_{\gph(h+F)}(\bar{x}+t_{k}u_{k},\bar{y}+t_{k}\hat v_{k}),
\end{gather*}
where
\if{
\begin{align*}
& \hat{v}_{k}:= v_{k}+ h(\bar{x}+t_{k}u_{k})/ t_{k}\\
&  \hat{u}^{*}_{k}:=\nabla h(\bar{x}+t_{k}u_{k})^{T} v^{*}_{k}-u^{*}_{k}
\end{align*}
}\fi
$\hat{v}_{k}:= v_{k}
%AK3/03/18
%+
-
Bu_{k}$, $\hat{u}^{*}_{k}:=u^{*}_{k}-B^{T} v^{*}_{k}$.
Observe that $\hat{v}_{k}\to0$ and $\hat{u}^{*}_{k}\to0$ as $k\to+\infty$,
which implies that
$(0,0)\in {\rm Cr}_{0}(h+F)(\bar{x},\bar{y})$.
By \cite[Theorem 3.2(2)]{Gfr11}, we can now find a $C^{1}$ perturbation $\tilde{h}$ with $\tilde{h}(\bar{x})=0$ and $\| \nabla \tilde{h}(\bar{x})\| = 0$ such that $F+h+\tilde{h}$ is not metrically subregular at $\bar{x}$ for $\bar{y}$.
Since $(h+\tilde{h})(\bar{x})=0$ and $\lip(h+\tilde h;\bx) =\|\nabla(h+\tilde{h})(\bar{x})\|=\norm{B}$, we conclude  that $\rad_{C^1}F(\bx\for\by)\le\norm{B}<\rgo+\varepsilon$.
Taking infimum in the last inequality over all $\varepsilon > 0$, we arrive at $\rad_{C^1}F(\bx\for\by)\le \rgo$.
%The inequalities \eqref{ThRadLipP6} and \eqref{ThRadLipP7} together with \eqref{8} prove \eqref{1}.
In view of \eqref{8}, this completes the proof of \eqref{1}.
%AK30/03/18
%AK14/04/18
%{\underline{$\rad_{Lip}F(\bx\for\by)\le \rgo$}.
%The inequality is a consequence of conditions \eqref{8} and \eqref{1}.}
\qed\end{proof}
\if{\HG{4/03/18.
I do not understand this remark: Think at the example $F\equiv 0$ in the introduction: The quantity \eqref{rgo} has nothing to do with the reciprocal of $\reg$. However, I conjecture that a relation of the form $rad \leq 1/reg$ always hold

\red{AK 30/03/18.
Is it OK now?}
}

\AK{10/03/18.
The remark should be rewritten.
The radius theorem places the subregularity property into the general scheme of regularity concepts if the `paradigm \eqref{radreg}' and `regularity modulus' are understood in a broader sense.
I think we should make a strong point about it.
(In the Introduction?)}
}\fi
\begin{remark}
Unlike the case of semismooth and $C^1$ perturbations, where Theorem~\ref{ThRadLip} establishes the exact formula for the radius, in the case of more general Lipschitz perturbations the theorem gives only lower and upper bounds for the respective radius, which, in view of Proposition~\ref{P3}(ii), differ by a factor of at most 2.
We do not know if these bounds are sharp.
Obtaining sharp bounds is an
interesting problem for future research
\end{remark}

By using the first inequality in \eqref{8} and Proposition~\ref{P3}(i),
one obtains additional bounds for the radii of subregularity, as stated in the following corollary.

\begin{corollary}%\label{ThRadLip}
Consider a mapping $F:\R^n \rightrightarrows \R^m$
with closed graph
and a point ${(\bx,\by)\in\gph F}$.
Then
%AK14/04/18
\begin{enumerate}
\item
{$\rad_{Lip}F(\bx\for\by)\le\rgo$};
\item
$\rad_{ss}F(\bx\for\by)=\rad_{C^1}F(\bx\for\by) \ge\rgd
{\ge\rg}$.
\end{enumerate}
\end{corollary}

\if{
%AK12/02/18
In view of
%AK3/03/18
parts (iii) and (iv) of Proposition~\ref{P3}
and the known representations of the radii of metric regularity and strong subregularity \cite[Theorems~6A.7, 4C.2, 6A.9 and 4E.1]{DonRoc14}, the next natural statement, providing relationships between the radius of subregularity and those of metric regularity and strong subregularity, is a consequence of Theorem~\ref{ThRadLip}.
The relationships recapture the well known fact that subregularity is the weakest of the three mentioned properties.
\begin{corollary}\label{P7}
Consider a mapping $F:\R^n \rightrightarrows \R^m$ with closed graph and a point $(\bx,\by)\in\gph F$.
Then
\[\rad_{Lip}F(\bx\for\by)\geq \max\big\{{\rm rad[MR]}F(\bx\for\by),{\rm rad[SSR]}F(\bx\for\by)\big\}.\]
\end{corollary}
The above corollary does not assume that $F$ is subregular at $\bx$ for $\by$.
If it is not, the conclusion is of course trivial as all the three radii in this case equal 0.
}\fi
\if{\HG{4/03/18.
This corollary holds trivially true because both metric regularity and strong subregularity imply subregularity.

\red{AK10/03/18.
Yes, this is what the sentence before Corollary~\ref{P7} says.}

But we also know
\[\rad_{Lip}F(\bx\for\by)\geq \rg\geq\max\big\{{\rm rad[MR]}F(\bx\for\by),{\rm rad[SSR]}F(\bx\for\by)\big\},\]
i.e. that the lower bound for the radius of subregularity is greater than the radii of metric regularity and strong subregularity.

\red{AK10/03/18.
Yes, this is Theorem~\ref{ThRadLip} plus Proposition~\ref{P3}.
They produce Corollary~\ref{P7}.}

\red{The radius theorem recaptures known results.
I think this is good.}
}}\fi
\if{\AK{23/01/18.
I am not quite happy with the long notations for the radii in the above theorem.
Could you think of something better?
In general, any remarks/objections concerning this form of the theorem?

It could make sense to introduce a special notation also for the last expression in \eqref{1} and try to obtain a reasonable interpretation for it making \eqref{1} an analogue of \eqref{radreg}.
Does the fact that this expression is positive constitute a kind of regularity of $F$?}}\fi

\if{Equation \eqref{1} can be written as
\begin{multline*}
\rad_{ss}F(\bx\for\by) =\rad_{C^1}F(\bx\for\by)
\\
= \inf_{B\in \R^{m\times n}} \bigl\{\|B\| \mid  B^Tv^*=u^\ast,\; -Bu=v,
\\
( u^*,v^*) \in\overline{N}_{\gph F}((\bx,\by);(u,v)),\; (u, v^*) \in\Sp_{\R^n}\times\Sp_{\R^m}\},
\end{multline*}
Multiplying the relations $B^Tv^*=u^\ast$  and $-Bu=v$ with $u^T$ and ${v^*}^T$, respectively, we obtain $u^T B^Tv^*=u^Tu^\ast$ and  $-{v^*}^TBu={v^*}^T$. By taking into account ${v^*}^TBu=u^TB^Tv^*$ this implies ${u^*}^Tu+{v^*}^Tv=0$ and we obtain
\begin{multline}\label{EqRadSSMod}
\rad_{ss}F(\bx\for\by) =\rad_{C^1}F(\bx\for\by)
\\
= \inf \bigl\{\|B\| \mid  B^Tv^*=u^\ast,\; -Bu=v,\;( u^*,v^*) \in\overline{N}_{\gph F}((\bx,\by);(u,v)),
\\
 {u^*}^Tu+{v^*}^Tv=0,\; (u, v^*) \in\Sp_{\R^n}\times\Sp_{\R^m}\}.
\end{multline}

\HG{24/05/18.
I think we should emphasize that the lower and upper bound for $\rad_{Lip}F(\bx\for\by)$ differ by a factor of at most 2 and are therefore sharp. In my opinion this is an important issue for conditioning.}
}\fi
%AK14/04/18
{In accordance with Theorem~\ref{ThRadLip}, condition $\rg>0$ guarantees that $F$ is metrically subregular at $\bx$ for $\by$ together with all its perturbations by Lipschitz continuous functions with small Lipschitz modulus, while condition $\rgo>0$ plays a similar role with respect to semismooth and $C^1$ perturbations of $F$.
In fact, both conditions correspond to certain regularity properties of $F$ at $\bx$ for $\by$ being stronger than conventional metric subregularity and, in view of Proposition~\ref{P3}(iv) and the well-known graphical derivative criterion for strong metric subregularity \cite[Theorem~4E.1]{DonRoc14}, weaker than strong metric subregularity.
\AK{25/05/18.
Could it make sense to give a name to the regularity property of $F$ at $(\bx,\by)$ determined by the inequality $\rg>0$?
\emph{Firmly subregular}?}
%AK28/02/18
%\begin{remark}%\label{P3}
Formula \eqref{1} agrees with the pattern of \eqref{radreg+} with
%the quantity \eqref{rgo}
$\rgo$
playing the role of
%(the reciprocal of)
the regularity `modulus' rg.
%\end{remark}
Note that the mentioned regularity properties, despite possessing certain stability with respect to small perturbations, are not `robust': they can be violated in a neighbourhood of the reference point $(\bx,\by)$; see the example in Section~\ref{S4}.
}

%\subsection{Euclidean and Frobenius norms}
Computing $\rgo$ using
\eqref{rgo} and
%AK14/04/18
%the second expression in
\eqref{DF2}
involves minimization over five parameters: four vectors $u,v,u^*,v^*$ and a matrix $B$.
The number of parameters could be reduced by eliminating the matrix if for given $u,v,u^*,v^*$,
satisfying ${u^*}^Tu={v^*}^Tv$ and $\|u\|=\|v^*\|_*=1$,
we were able to solve analytically the problem
\begin{gather*}%\label{27}
\min_{B\in L(\R^n,\R^m)} \norm{B}\mbox{ subject to } B^Tv^*=u^\ast,\; Bu=v,
\end{gather*}
where $\norm{B}$ denotes
%some
the
operator norm.

Currently we know
%an
the
explicit solution to this problem only for the Frobenius norm $\norm{B}_F$
in the case when $\R^n$ and $\R^m$ are considered with the Euclidean norms.
Specifically, the next proposition deals with the convex constrained optimization problem
\begin{equation}\label{EqFrobMinProbl}
\min \frac 12\norm{B}_F^2 \mbox{ subject to } B^Tv^*=u^\ast,\;
Bu=v.
\end{equation}

\begin{proposition}\label{PrFrob}
Let vectors $u,u^*\in\R^n$ and $v,v^*\in\R^m$ satisfy conditions
\begin{gather}\label{PrFrob-1}
{u^*}^Tu={v^*}^Tv,\quad\norm{u}_2=\norm{v^\ast}_2=1, \end{gather}
where $\norm{\cdot}_2$ denotes the Euclidean norm.
The unique minimizer of the problem \eqref{EqFrobMinProbl} is given by
the matrix
\[\overline{B}:= v^\ast{u^\ast}^T
+
vu^T-({u^\ast}^Tu)v^\ast u^T,\]
and in this case,
\begin{equation}\label{EqFrobMinProb2}
\norm{\overline{B}}_F^2= \norm{u^\ast}_2^2+\norm{v}_2^2-\big({u^\ast}^Tu\big)^2.
\end{equation}
\end{proposition}
\begin{proof}
The feasibility of $\overline{B}$ in the problem \eqref{EqFrobMinProbl} can be shown by straightforward calculations:
\begin{gather*}
\overline{B}^Tv^\ast=\left(u^\ast {v^\ast}^T+uv^T -({u^\ast}^Tu)u{v^\ast}^T\right)v^\ast
=\norm{v^\ast}_2^2 u^\ast+\left(v^Tv^\ast-({u^\ast}^Tu) \norm{v^\ast}_2^2\right)u =u^\ast,\\
\overline{B}u=(v^\ast{u^\ast}^T+vu^T-({u^\ast}^Tu)v^\ast u^T)u
=({u^\ast}^Tu)v^\ast+\norm{u}_2^2 v -({u^\ast}^Tu)\norm{u}_2^2 v^\ast=v.
\end{gather*}
Furthermore, $\overline{B}$ satisfies the first-order KKT condition for problem \eqref{EqFrobMinProbl} with multipliers $\eta=-u^\ast$ and $\eta^\ast=({u^\ast}^Tu)v^\ast-v$; indeed:
\[\overline{B}+v^\ast\eta^T+\eta^\ast u^T =v^\ast{u^\ast}^T+vu^T-({u^\ast}^Tu)v^\ast u^T -v^\ast {u^*}^T+(({u^\ast}^Tu)v^\ast-v)u^T=0.\]
Since \eqref{EqFrobMinProbl} is a strictly convex program, our claim about the optimality of $\overline{B}$ is verified.
%In order to complete the proof observe that
Next we show \eqref{EqFrobMinProb2}.
\begin{align*}\nonumber
\norm{\overline{B}}_F^2 &=\tr\left(\overline{B}^T\overline{B}\right)
\\\nonumber
&=\tr\Big(\big(u^\ast{v^\ast}^T+uv^T -({u^\ast}^Tu)u{v^\ast}^T\big)
\big(v^\ast{u^\ast}^T+vu^T - ({u^\ast}^Tu)v^\ast u^T\big)\Big)
\\\nonumber
&=\tr\big(\norm{v^\ast}_2^2 u^\ast{u^\ast}^T +({v^*}^Tv)u{u^*}^T-\norm{v^\ast}_2^2({u^\ast}^Tu)u {u^\ast}^T
\\\nonumber
&\quad+({v^\ast}^Tv) u^\ast u^T +\norm{v}_2^2 uu^T -({u^\ast}^Tu)({v^\ast}^Tv)uu^T
\\\nonumber
&\quad-\norm{v^\ast}_2^2({u^\ast}^Tu)u{u^\ast}^T -({u^\ast}^Tu)({v^*}^Tv)uu^T +({u^\ast}^Tu)^2\norm{v^\ast}_2^2uu^T\big)
\\\nonumber
&=\tr\Big(\norm{v^\ast}_2^2u^\ast{u^\ast}^T +\big(({v^*}^Tv)-2\norm{v^\ast}_2^2({u^\ast}^Tu)\big) u{u^\ast}^T
\\\nonumber
&\quad+({v^\ast}^Tv) u^\ast u^T +\big(\norm{v}_2^2 -({u^\ast}^Tu) \big(2({v^\ast}^Tv) -({u^\ast}^Tu)\norm{v^\ast}_2^2\big)\big)uu^T\Big)
\\\nonumber
&=\norm{v^\ast}_2^2\norm{u^*}_2^2 +\big(({v^*}^Tv)-2\norm{v^\ast}_2^2({u^\ast}^Tu)\big) ({u^\ast}^T u)
\\\nonumber
&\quad+({v^\ast}^Tv)({u^\ast}^T u) +\big(\norm{v}_2^2 -({u^\ast}^Tu) \big(2({v^\ast}^Tv)-({u^\ast}^Tu)\norm{v^\ast}_2^2\big)\big)\norm{u}_2^2. \end{align*}
In view of \eqref{PrFrob-1}, we get
\begin{align*}\nonumber
\norm{\overline{B}}_F^2&=\norm{u^*}_2^2 -({u^\ast}^Tu)^2
+({v^\ast}^Tv)({u^\ast}^Tu) +\norm{v}_2^2 -({u^\ast}^Tu)({v^\ast}^Tv)
\\\nonumber
&=\norm{u^*}_2^2 -({u^\ast}^Tu)^2
+\norm{v}_2^2.
\end{align*}
The proof is complete.
\qed\end{proof}

%AK3/03/18
In view of the above proposition, in the Euclidean space setting the following analogue (upper bound) of the quantity \eqref{rgo} can be used for estimating the radius of subregularity:
\begin{multline}\label{rgdag}
\rgdag:=\inf \Bigl\{\sqrt{\norm{u^\ast}_2^2+\norm{v}_2^2-({u^*}^Tu)^2} \mid
(u^*,v)\in\widehat{D}F(\bx,\by)(u,v^*),\\
{u^*}^Tu={v^*}^Tv,\;
\|u\|_2=\|v^*\|_2=1
\Bigr\}.
\end{multline}
\if{
\HG{24/05/18.
Similar to my suggestion above we could rename $\rgdag$ by $\overline{{\rm rg}}^\circ[F](\bx,\by)$.}
}\fi
The next proposition provides relationships between this new quantity and \eqref{rgo}.
\begin{proposition}\label{PropBndRadSS}
Consider a mapping $F:\R^n \rightrightarrows \R^m$
and a point $(\bx,\by)\in\gph F.$
If both $\R^m$ and $\R^n$ are equipped with the Euclidean norm, then
%the bound
\begin{gather*}
%\rad_{ss} F(\bx\for\by)
\rgo
\leq \rgdag\leq\sqrt{2}
%\xi_{ss}
\,\rgo.
\end{gather*}
%also holds.
\end{proposition}

\if{\AK{1/03/18.
The infimum in the middle looks like another `regularity' constant.
Should it be introduced?}}\fi

\begin{proof}
The first inequality is a consequence of Proposition~\ref{PrFrob} since
the spectral norm of a matrix, i.e. the operator norm with respect to the Euclidean norm, is always less than or equal to the Frobenius norm.
The second inequality follows immediately from the estimate $\sqrt{\norm{u^\ast}_2^2+\norm{v}_2^2-({u^*}^Tu)^2}\leq \sqrt{2}\max\{\norm{u^\ast}_2,\norm{v}_2\}$.
\qed\end{proof}

\if{
\AK{22/02/18.
Should the statements for the radii be formulated explicitly as corollaries?}
}\fi

\begin{corollary}%\label{ThRadLip}
If both $\R^m$ and $\R^n$ are equipped with the Euclidean norm, then
\begin{enumerate}
\item
%AK30/03/18
$\rad_{Lip}F(\bx\for\by)\le\rgdag$;
\item
%AK24/03/18
$\frac{1}{\sqrt{2}}\rgdag \le\rad_{ss}F(\bx\for\by)=\rad_{C^1}F(\bx\for\by)
\le\rgdag$.
\end{enumerate}
\end{corollary}

\section{Applications to constraint systems}\label{S4}

\AK{6/06/18.
Should $G$ be lowercase in accordance with the convention in the Preliminaries?}

Consider the constraint system
\begin{gather}\label{4.1}
x\in D,\quad
g(x)\in K,
\end{gather}
where $D\subset\R^n$, $K\subset\R^m$, $g:\R^n\to\R^m$, $\bx\in D$ and $g(\bx)\in K$.
\if{\AK{24/01/18:
I have changed the notations in Jiri's notes to make them comply with those in the preceding section.
Shall we change the notations in the preceding section instead?}}\fi
The inclusions \eqref{4.1} can be equivalently written as $0 \in F(x)$, where
\begin{gather}\label{4.2}
F(x):=
\begin{cases}
K-g(x)&\mbox{if } x\in D,
\\
\emptyset&\mbox{otherwise}.
\end{cases}
\end{gather}
Observe that $\by:=0\in F(\bx)$.

Before we apply our theory to the \SVM\ $F$ given by (\ref{4.2}),
%AK6.06.18
%note that,
{we recall two facts used in the proof of Proposition~\ref{P2} below.
The first one comes from \cite[Proposition~3.2]{YeZho18}.

\begin{lemma}\label{L14}
Given two sets $A_{1}\subset \mathbb{R}^{n}$ and $A_{2}\subset \mathbb{R}^{m}$, a point $\bar{x}=(\bar{x}_{1}, \bar{x}_{2})\in A_{1} \times A_{2}$ and a direction $u=(u_{1}, u_{2})\in\mathbb{R}^{n} \times \mathbb{R}^{m}$, one has the inclusion
\[
\overline{N}_{A_{1} \times A_{2}}(\bar{x};u) \subset \overline{N}_{A_{1}}(\bar{x}_{1}; u_{1}) \times \overline{N}_{A_{2}}(\bar{x}_{2}; u_{2}).
\]
This inclusion becomes equality provided that either $A_{1}$ is directionally regular at $\bar{x}_{1}$ in the direction $u_{1}$ or $A_{2}$ is directionally regular at $\bar{x}_{2}$ in the direction $u_{2}$.
\end{lemma}
Next we need \cite[formula~(2.4)]{GfrOut16.2} for computing the directional limiting coderivative.
\begin{lemma}\label{L15}
Consider the mapping $F:=f_1+F_2$, where $F_2:\R^n\rightrightarrows\R^m$ has closed graph and $f_1:\R^n\to\R^m$ is continuously differentiable at $\bx\in\dom F_2$.
Given a $\by\in F(\bx)$, a pair $(u,v)\in\R^n\times\R^m$ and a $y^*\in\R^m$, it holds
\begin{gather}\label{L15-1}
\overline{D}^*F((\bx,\by);(u,v))(y^*)=\nabla f_1(\bx)^Ty^* +\overline{D}^*F_2((\bx,\by-f_1(\bx));(u,v-\nabla f_1(\bx)u))(y^*).
\end{gather}
\end{lemma}
Since formula \eqref{L15-1} was given in \cite{GfrOut16.2} without proof, we provide here its short proof for completeness.
\begin{proof}
In view of the differentiability of $f_1$ near $\bx$, we have
\begin{gather*}%\label{L15-1}
D^*F(x,y)({y^*}')=\nabla f_1(x)^T{y^*}' +D^*F_2(x,y-f_1(x))({y^*}')
\end{gather*}
for all $x$ near $\bx$ and all $y\in F(x)$ and ${y^*}'\in\R^m$.
By the definition of the directional limiting coderivative and using the above equality and continuous differentiability of $f_1$, we have
\begin{align*}%\label{L15-1}
\overline{D}^*F&((\bx,\by);(u,v))(y^*) =\Limsup_{\substack{u'\to u,\, v'\to v\\{y^*}'\to y^*,\,t\searrow0}}D^*F(\bx+tu',\by+tv')({y^*}')
\\
&=\Limsup_{\substack{u'\to u,\, v'\to v\\{y^*}'\to y^*,\,t\searrow0}}\left[\nabla f_1(\bx+tu')^T{y^*}' +D^*F_2(\bx+tu',\by+tv'-f_1(\bx+tu'))({y^*}')\right]
\\
&=\nabla f_1(\bx)^T{y^*} +\Limsup_{\substack{u'\to u,\, v'\to v\\{y^*}'\to y^*,\,t\searrow0}} D^*F_2(\bx+tu',\by-f_1(\bx)+t(v'-\nabla f_1(\bx)u'-o(t)/t))({y^*}')
\\
&=\nabla f_1(\bx)^T{y^*} +\overline{D}^*F_2((\bx,\by-f_1(\bx));(u,v-\nabla f_1(\bx)u))(y^*).
\end{align*}
The proof is complete.
\qed\end{proof}
}

Below we compute the quantities crucial for determining estimates for the radii of metric subregularity of $F$.

\begin{proposition}\label{P2}
Suppose that the sets $D$ and $K$ are closed, and either $D$ is directionally regular at $\bx$ or $K$ is directionally regular at $g(\bx)$.
Suppose also that $g$ is continuously differentiable near $\bx$.
Then
\begin{align}\notag
\rg= &\inf \bigl\{\max\{\norm{u^*}_*,\norm{v}\} \mid
u^*+\nabla g(\bx)^Tv^*\in\overline{N}_{D}(\bx;u),
\\\label{P2-1}
&\qquad-v^*\in\overline{N}_{K}(g(\bx);v
%AK13/03/18
%-
+
\nabla g(\bx)u),\;
\|u\|=\|v^*\|_*=1\bigr\},
\\\notag
\rgo
= &\inf\bigl\{\|B\| \mid
%AK3/03/18
B\in L(\R^n,\R^m),\;
%AK13/03/18
%(B^T+\nabla g(\bx)^T)
(B+\nabla g(\bx))^T
v^*\in\overline{N}_{D}(\bx;u),
\\\label{P2-2}
&\qquad-v^*
\in\overline{N}_{K}(g(\bx);(B
%-
+
\nabla g(\bx))u),\;
%\mbox{for some}
%AK27/02/18
%(u, v^*) \in\Sp_{\R^n}\times\Sp_{\R^m}
\|u\|=\|v^*\|_*=1
\bigr\},
\\\notag
{\rgh=} &{\inf \bigl\{\norm{u^*}_*+\norm{v} \mid
u^*+\nabla g(\bx)^Tv^*\in\overline{N}_{D}(\bx;u),}
\\\label{P2-3}
&{\qquad-v^*\in\overline{N}_{K}(g(\bx);v
+
\nabla g(\bx)u),\;
\|u\|=\|v^*\|_*=1\bigr\}.}
\end{align}
\end{proposition}

\if{
\AD{17/05/18:
What does it mean ``directionally regular"?}
\AK{25.01.18:
The definitions of directional regularity and directional limiting normal cone to be added.}

\todo[inline]{AK 18/03/17:
\cite[formula (2.4)]{GfrOut16.2} plays a crucial role in the proof below.
I suggest providing its direct proof and also reformulating it in a way more appropriate for the proof.}
}\fi

\begin{proof}
Observe that $F(x)=H(x)-g(x)$ $(x\in\R^n)$, where $\gph H=D\times K$.
%AK6.06.18
%In view of the directional regularity of either $D$ or $K$, one has
{By Lemma~\ref{L14},}
\begin{gather*}%\label{4.1}
\overline{N}_{D\times K}((\bx,g(\bx));(u,v)) =\overline{N}_{D}(\bx;u)\times\overline{N}_{K}(g(\bx);v)
\end{gather*}
for all $(u,v)\in\R^n\times\R^m$.
%cf. \cite[Proposition~3.2]{YeZho18}.
Hence, by virtue of
%AK6.06.18
%\cite[formula~(2.4)]{GfrOut16.2},
{Lemma~\ref{L15},}
we obtain
\begin{multline*}%\label{4.1}
\overline{D}^*F((\bx,0);(u,v))(v^*)
%AK13/03/18
=\overline{D}^*H((\bx,g(\bx));(u,v+\nabla g(\bx)u))(v^*) -\nabla g(\bx)^Tv^*
\\=
\begin{cases}
\overline{N}_{D}(\bx;u)-\nabla g(\bx)^Tv^* &\mbox{if } -v^*\in\overline{N}_{K}(g(\bx);v
%-
+
\nabla g(\bx)u),
\\
\es & \mbox{otherwise}.
\end{cases}
\end{multline*}
%AK4/03/18
It follows
from the representations \eqref{pdd}, \eqref{DF} and \eqref{DF2}
that
\begin{align*}%\label{pdd}
\widehat{D}F(\bx,0)(u,v^*)= &\bigl(\overline{N}_{D}(\bx;u)-\nabla g(\bx)^Tv^*\bigr)
\times\bigl\{v\in\R^m\mid
\\
&\qquad-v^*\in\overline{N}_{K}(g(\bx);v
%-
+
\nabla g(\bx)u)\bigr\}, \;(u,v^*)\in\R^n\times(\R^m)^*,
\\
\widehat{\mathfrak{D}}F(\bx,0)=
&\big\{(u^*,v)\in(\R^n)^*\times\R^m\mid u^*+\nabla g(\bx)^Tv^*\in\overline{N}_{D}(\bx;u),
\\
&\qquad-v^*\in\overline{N}_{K}(g(\bx);v
%-
+
\nabla g(\bx)u),\;
\|u\|=\|v^*\|_*=1\big\},
\\%\notag
\widehat{\mathfrak{D}}^\circ F(\bx,0)=
&\big\{(B^Tv^*,Bu)\mid
%AK13/03/18
%(B^T+\nabla g(\bx)^T)
(B+\nabla g(\bx))^T
v^*\in\overline{N}_{D}(\bx;u),
\\
&\qquad -v^*\in\overline{N}_{K}(g(\bx);(B
%-
+
\nabla g(\bx))u),\;\|u\|=\|v^*\|_*=1%,\; B\in L(\R^n,\R^m)
\big\}.
\end{align*}
%for all $(u,v)\in T_{\gph F}(\bx,0)$.
Substituting the last two expressions into the definitions \eqref{rg}, \eqref{rgo} and \eqref{rgh} leads to the representations \eqref{P2-1}, \eqref{P2-2} and \eqref{P2-3}, respectively.
\qed\end{proof}
\AK{25/01/18:
%I seem to have got lost in the signs.
%Please check.

%Can the minuses in \eqref{P2-1} be dropped?

%Can the tangent cones in the proof be dropped?
%They seems to be implicitly implied by the definitions of directional limiting normal cones.

It would be good to formulate some meaningful interpretation of
%the last expression in
\eqref{P2-1},
\eqref{P2-2},
\eqref{P2-3}
in terms of a kind of `regularity' of the constraint system.}
\if{\AK{4/03/18:
I seem to have a problem with the signs again.
In the previous version, $\nabla g(\bx)$ entered both parts of \eqref{P2-2} with the same sign, and it was nice.
I have spoiled the picture.
The same with representation \eqref{P2-1}.
}}\fi
\if{
\begin{remark}
By \cite[Proposition~7.6]{Iof17}, metric subregularity of the feasibility mapping \eqref{4.2} at $\bx$ for 0 is equivalent to that of the following `extended' feasibility mapping $\widetilde{F}:\R^n \rightrightarrows \R^{m+n}$ (at the same point):
\begin{gather*}
\widetilde{F}(x):=
\begin{bmatrix}
g(x)-K
\\
x-D
\end{bmatrix},\quad
x\in\R^n.
\end{gather*}
\end{remark}

For a constraint system
\begin{gather*}%\label{4.1}
g(x)\in K,
\end{gather*}
which is a particular case of \eqref{4.1} with $D=\R^n$, we can formulate the following corollary using the simplified feasibility mapping $F(x):=g(x)-K$, $x\in\R^n$.}\fi

The next
corollary is a consequence of Proposition~\ref{P2} and Theorem~\ref{ThRadLip}.
It gives estimates for the radii
of metric subregularity of $F$ at $\bx$ for 0, or equivalently, of calmness of the corresponding solution mapping
\begin{gather*}%\label{4.1}
S(y):=\{x\in D\mid
g(x)
%AK5/02/18
%-
+
y\in K\},\quad y\in\R^m
\end{gather*}
at $0$ for $\bx$.
%It is a consequence of Theorem~\ref{T2}.
%Further, by \cite[Proposition~7.6]{Iof17} it is equivalent to the subtransversality of the set system $\{g\iv(K),D\}$ at $\bx$.

%AK4/03/18
\begin{corollary}%\label{P2-1}
Under the assumptions of Proposition~\ref{P2},
\begin{align*}%\label{EqRadLip}
&{\inf \bigl\{\max\{\norm{u^*}_*,\norm{v}\} \mid
u^*+\nabla g(\bx)^Tv^*\in\overline{N}_{D}(\bx;u),}
\\
&{\qquad-v^*\in\overline{N}_{K}(g(\bx);v+\nabla g(\bx)u),\;
\|u\|=\|v^*\|_*=1\bigr\}}
\\%\label{P2-1}
\le\rad_{Lip}F(\bx\for0)\leq
&\inf \bigl\{\norm{u^*}_*+\norm{v}\mid
u^*+\nabla g(\bx)^Tv^*\in\overline{N}_{D}(\bx;u),
\\%\label{P2-1}
&\qquad-v^*\in\overline{N}_{K}(g(\bx);v+\nabla g(\bx)u),\;
\|u\|=\|v^*\|_*=1\bigr\},
\\
\rad_{ss}F(\bx\for0)=& \rad_{C^1}F(\bx\for0)
\\=&\inf\bigl\{\|B\| \mid
%AK3/03/18
B\in L(\R^n,\R^m),
(B+\nabla g(\bx))^T v^*\in\overline{N}_{D}(\bx;u),
\\
&\qquad-v^*
\in\overline{N}_{K}(g(\bx);(B+\nabla g(\bx))u),\;
%\mbox{for some}
%AK27/02/18
%(u, v^*) \in\Sp_{\R^n}\times\Sp_{\R^m}
\|u\|=\|v^*\|_*=1
\bigr\}.
\end{align*}
\end{corollary}
The particular case of the constraint system
\begin{gather*}%\label{4.1}
g(x)\in K
\end{gather*}
corresponds to taking $D=\R^n$ in \eqref{4.1},
while the ``feasibility'' mapping takes the form
\begin{gather}\label{4.34}
F(x):=K-g(x),\quad x\in\R^n.
\end{gather}
Assuming that $g(\bx)\in K$, we again have $\by:=0\in F(\bx)$.
Note that the set $D=\R^n$ is automatically directionally regular at any point.
\begin{corollary}\label{C14}
Suppose that the set $K$ is closed, $g$ is continuously differentiable near $\bx$ and $F$ is given by \eqref{4.34}.
Then
\begin{align*}\notag
\rg= &\inf \bigl\{\max\{\norm{u^*}_*,\norm{v}\} \mid
u^*+\nabla g(\bx)^Tv^*=0,
\\%\label{P2-1}
&\qquad-v^*\in\overline{N}_{K}(g(\bx);v+\nabla g(\bx)u),\;
\|u\|=\|v^*\|_*=1\bigr\},
\\\notag
\rgo=&
%AK3/03/18
\inf\bigl\{\|B\| \mid
B\in L(\R^n,\R^m),\;
(B+\nabla g(\bx))^T v^*=0,
\\
&\qquad-v^*
\in\overline{N}_{K}(g(\bx);(B+\nabla g(\bx))u),\;
%\mbox{for some}
%AK27/02/18
%(u, v^*) \in\Sp_{\R^n}\times\Sp_{\R^m}
\|u\|=\|v^*\|_*=1
\bigr\},
\\\notag
{\rgh=} &{\inf \bigl\{\norm{u^*}_*+\norm{v} \mid
u^*+\nabla g(\bx)^Tv^*=0,}
\\%\label{P2-3}
&{\qquad-v^*\in\overline{N}_{K}(g(\bx);v
+
\nabla g(\bx)u),\;
\|u\|=\|v^*\|_*=1\bigr\}.}
\end{align*}
\end{corollary}
\begin{corollary}%\label{P2-1}
Under the assumptions of Corollary~\ref{C14},
\begin{align*}%\label{EqRadLip}
&{\inf \bigl\{\max\{\norm{u^*}_*,\norm{v}\} \mid
u^*+\nabla g(\bx)^Tv^*=0,}
\\
&{\qquad-v^*\in\overline{N}_{K}(g(\bx);v+\nabla g(\bx)u),\;
\|u\|=\|v^*\|_*=1\bigr\}}
\\%\label{P2-1}
\le\rad_{Lip}F(\bx\for0)\leq
&\inf \bigl\{\norm{u^*}_*+\norm{v}\mid
u^*+\nabla g(\bx)^Tv^*=0,
\\%\label{P2-1}
&\qquad-v^*\in\overline{N}_{K}(g(\bx);v+\nabla g(\bx)u),\;
\|u\|=\|v^*\|_*=1\bigr\},
\\
\rad_{ss}F(\bx\for0)=& \rad_{C^1}F(\bx\for0)
\\=&\inf\bigl\{\|B\| \mid
%AK3/03/18
B\in L(\R^n,\R^m),
(B+\nabla g(\bx))^Tv^*=0,
\\
&\qquad-v^*
\in\overline{N}_{K}(g(\bx);(B+\nabla g(\bx))u),\;
%\mbox{for some}
%AK27/02/18
%(u, v^*) \in\Sp_{\R^n}\times\Sp_{\R^m}
\|u\|=\|v^*\|_*=1
\bigr\}.
\end{align*}
\end{corollary}

\if{\AK{25.01.18:
Is directional regularity of $K$ needed in the Corollary?}}\fi

\if{
\begin{remark}%\label{P2}
In the proof of Proposition~\ref{P2}, using representation \eqref{P2P-1} of the directional limiting cone to the graph of $F$, quantity \eqref{rgo} was computed.
The same way, one can compute quantities \eqref{rgd} and \eqref{rgdag}:
\end{remark}
}\fi

%\section{Examples}\label{S5}

\if{\AK{4/03/18
The section should be updated.
It would be good to compute all constants.
}}\fi

%AK15/04/18
{
%AK1705/18
%In this section
Now
we illustrate the
%main
above
results
%of the paper
by examples.

%AK4/06/18
\begin{example}%\label{E1}
%First,
It is easy to check by direct computation that, for the zero mapping $f:\R\to\R$ (that is $f(x)=0$ for all $x\in\R$) considered in Example~\ref{E2}, it holds rg$f(\bx,0)=$rg$^\circ f(\bx,0)=0$ for any $\bx\in\R$.
Hence, by Theorem~\ref{ThRadLip}, $$\rad_{Lip}f(\bx\for0)=
\rad_{ss}f(\bx\for0)=\rad_{C^1}f(\bx\for0)=0,$$ which of course agrees with the observation made in Example~\ref{E2}.
\qed\end{example}

Next we consider a couple of more involved examples.}

%\subsection{Constrained system}

%We start with an example having the structure of the constrained system considered in the previous section.
%It is a modification of \cite[Example~3]{GfrOut16}.

\begin{example}%\label{E1}
Let the mapping $F:\R^2\rightrightarrows\R^2$ be defined as follows:
\begin{gather}\label{E4-0}
F(x)=
\begin{cases}
x-K&\mbox{if } x\in D,
\\
\emptyset&\mbox{otherwise},
\end{cases}
\end{gather}
where $D=\{(x_1,x_2)\in\R^2\mid|x_2|\le x_1\}$ and $K=\{(x_1,x_2)\in\R^2_+\mid x_1 x_2=0\}$ is the ``complementary angle''.
The mapping \eqref{E4-0} can be considered as a special case of \eqref{4.2} with $g$ being the identity mapping.
We have $(\bx,\by)\in\gph F$ with
$\bx=\by=0\in\R^2$.

Since $F$ is polyhedral, it is metrically subregular at $\bx$ for $\by$.
%(cf. \cite{Rob81}).
%AK30/03/18
{At the same time, it is not strongly subregular at $\bx=0$ for $\by=0$ as 0 is not an isolated point of $F\iv(0)=D\cap K={\R_+\times\{0\}}$.}
Next we employ the tools of Section~\ref{main} to demonstrate that the metric subregularity of $F$ is preserved if it is perturbed by functions from the classes $\mathcal{F}_{Lip}$, $\mathcal{F}_{ss}$, $\mathcal{F}_{C^1}$ with sufficiently small Lipschitz moduli at $\bx$, and compute the respective radii.

In the current setting, formulas \eqref{P2-1}, \eqref{P2-3} and \eqref{P2-2} take, respectively, the following form:
\begin{align}\notag
\rg= &\inf \bigl\{\max\{\norm{u^*}_*,\norm{v}\} \mid
u^*+v^*\in\overline{N}_{D}(\bx;u),
\\\label{4.35}
&\hspace{3cm} -v^*\in\overline{N}_{K}(\bx;u+v),\;
\|u\|=\|v^*\|_*=1\bigr\},
\\\notag
\rgh= &\inf \bigl\{\norm{u^*}_*+\norm{v}\mid
u^*+v^*\in\overline{N}_{D}(\bx;u),
\\\label{4.36}
&\hspace{3cm} -v^*\in\overline{N}_{K}(\bx;u+v),\;
\|u\|=\|v^*\|_*=1\bigr\},
\\\notag
\rgo
= &\inf\bigl\{\|B\| \mid
B\in L(\R^n,\R^m),\;
(B+I)^Tv^*\in\overline{N}_{D}(\bx;u),
\\\label{4.37}
&\hspace{3cm} -v^*\in\overline{N}_{K}(\bx;(B+I)u),\;
\|u\|=\|v^*\|_*=1
\bigr\},
\end{align}
where $I$ denotes the identity mapping.

The directional limiting normal cones to $D$ and $K$ involved in \eqref{4.35}, \eqref{4.36} and \eqref{4.37} can be easily computed.
For any $u=(u_1,u_2)\in\R^2$, we have
\begin{align}\label{E4-2}
\overline{N}_{D}(\bx;u)=N_D(u)&=
\begin{cases}
\{(\xi_1,\xi_2)\mid \xi_1+|\xi_2|\le0\}&\mbox{if } u_1=u_2=0,
\\
\{(\xi_1,\xi_2)\mid \xi_1=-\xi_2\le0\}&\mbox{if } u_1=u_2>0,
\\
\{(\xi_1,\xi_2)\mid \xi_1=\xi_2\le0\}&\mbox{if } u_1=-u_2>0,
\\
\{(0,0)\}&\mbox{if } |u_2|<u_1,
\\
\emptyset&\mbox{otherwise},
\end{cases}
\\\label{E4-3}
\overline{N}_{K}(\bx;u)=N_K(u)&=
\begin{cases}
\R^2_-&\mbox{if } u_1=u_2=0,
\\
\{0\}\times\R&\mbox{if } u_1>0,\;u_2=0,
\\
\R\times\{0\}&\mbox{if } u_1=0,\;u_2>0,
\\
\emptyset&\mbox{otherwise}.
\end{cases}
\end{align}
Of course, only the points producing nonempty cones are of interest.
Besides, in accordance with \eqref{4.35}, \eqref{4.36} and \eqref{4.37}, one only needs to compute normals to $D$ at nonzero points; thus, the first case in \eqref{E4-2} can be excluded.
These observations leave us with three cases in \eqref{E4-2} (cases 2--4) and three cases in \eqref{E4-3} (cases 1--3), which produce 9 combinations.

Let vectors $u=(u_1,u_2)$, $v=(v_1,v_2)$, $u^*=(u_1^*,u_2^*)$ and $v^*=(v_1^*,v_2^*)$ be such that
\begin{gather}\label{E4-7}
u^*+v^*\in\overline{N}_D(u),\quad
%AK20/03/18
%\red{-}
-v^*\in\overline{N}_K(u+v),\quad \norm{u}=\norm{v^*}_*=1.
\end{gather}

Case 4 in \eqref{E4-2} leads to $u^*+v^*=0$, and consequently, $\norm{u^*}_*=\norm{v^*}_*=1$.
Similarly, case 1 in \eqref{E4-3} leads to $u+v=0$, and consequently, $\norm{v}=\norm{u}=1$.
Thus, in each of these two cases, $\max\{\norm{u^*}_*,\norm{v}\}\ge1$.

In all four combinations of the remaining cases 2 and 3 in \eqref{E4-2} and cases 2 and 3 in \eqref{E4-3}, we have $|u_1|=|u_2|$, $|u_1^*+v_1^*|=|u_2^*+v_2^*|$, and either $|v_1|=|u_1|$ and $v_2^*=0$, or $|v_2|=|u_2|$ and $v_1^*=0$.
Further analysis of these combinations depends on the type of the norm on $\R^2$ used in the above relations.
Let $\R^2$ be equipped with the $l_p$ ($1\le p\le+\infty$) norm:
$\norm{(u_1,u_2)}_p =\left(|u_1|^p+|u_2|^p\right)^{\frac{1}{p}}$ for all $(u_1,u_2)\in\R^2$.
Recall the usual convention: $\norm{(u_1,u_2)}_\infty=\max\{|u_1|,|u_2|\}$.

Since $\norm{u}=1$, we have $|u_1|=|u_2|=2^{-\frac{1}{p}}$
Since $\norm{v^*}_*=1$, we also have either ${|v_1|=2^{-\frac{1}{p}}}$, $|v_1^*|=1$ and $|u_1^*+v_1^*|=|u_2^*|$, or $|v_2|=2^{-\frac{1}{p}}$, $|v_2^*|=1$ and $|u_2^*+v_2^*|=|u_1^*|$.
In both cases, we obtain $\norm{v}\ge 2^{-\frac{1}{p}}$, $|u_1^*|+|u_2^*|\ge1$, and consequently,
%30/03/18
{using the standard relationship between $l_q$ and $l_1$ norms,}
\AK{17/05/18.
A reference?}
$\norm{u^*}_*=\norm{u^*}_q\ge {2^{-\frac{1}{p}}\norm{u^*}_1=} 2^{-\frac{1}{p}}(|u_1^*|+|u_2^*|)\ge 2^{-\frac{1}{p}}$, where $q>0$ and $\frac{1}{p}+\frac{1}{q}=1$.
Thus, $\max\{\norm{u^*}_*,\norm{v}\}\ge 2^{-\frac{1}{p}}$.
Since $2^{-\frac{1}{p}}\le1$, taking into account the estimates for case 4 in \eqref{E4-2} and case 1 in \eqref{E4-3}, we conclude that $\rg\ge 2^{-\frac{1}{p}}$.

Moreover, the above estimate is attained.
Indeed, take $u:=\big(2^{-\frac{1}{p}},2^{-\frac{1}{p}}\big)$, ${v:=\big(0,-2^{-\frac{1}{p}}\big)}$, $u^*:=\big(-\frac{1}{2},-\frac{1}{2}\big)$ and $v^*:=(0,1)$ to satisfy all the conditions in \eqref{E4-7}.
Then $\norm{v}_p=\norm{u^*}_q=2^{-\frac{1}{p}}$.
It follows that $\rg=2^{-\frac{1}{p}}$.
Observe that ${v=Bu}$ and $u^*=B^Tv^*$, where
$B=
\begin{pmatrix}
0&0\\-\frac{1}{2}&-\frac{1}{2}
\end{pmatrix}
$.
Obviously, $\norm{B}=2^{-\frac{1}{p}}$.
Comparing formulas \eqref{4.35} and \eqref{4.37} and taking into account Proposition~\ref{P3}(i), we conclude that
%30/03/18
${\rg=\rgo=\rgd}
=2^{-\frac{1}{p}}$.
At the same time, by \eqref{4.36}, ${\rgh\le2^{\frac{1}{q}}}$.
In accordance with Theorem~\ref{ThRadLip},
%30/03/18
$${\rad_{Lip}F(\bx\for\by)=}
\rad_{ss}F(\bx\for\by)=\rad_{C^1}F(\bx\for\by) =2^{-\frac{1}{p}}.$$
%while $2^{-\frac{1}{p}}\le\rad_{Lip}F(\bx\for\by) \le2^{\frac{1}{q}}$.

In the particular cases of interest, we have the following
%30/03/18
%estimates
{values}
for the radii:
\begin{itemize}
\item
$p=1$:
$
%30/03/18
{\rad_{Lip}F(\bx\for\by)=}
\rad_{ss}F(\bx\for\by)=\rad_{C^1}F(\bx\for\by) =\frac{1}{2}$;
%\\ $\frac{1}{2}\le\rad_{Lip}F(\bx\for\by) \le1$;
\item
$p=2$:
$
%30/03/18
{\rad_{Lip}F(\bx\for\by)=}
\rad_{ss}F(\bx\for\by)=\rad_{C^1}F(\bx\for\by) =\frac{1}{\sqrt{2}}$;
%\\ $\frac{1}{\sqrt{2}}\le\rad_{Lip}F(\bx\for\by) \le\sqrt{2}$;
\item
$p=+\infty$:
$
%30/03/18
{\rad_{Lip}F(\bx\for\by)=}
\rad_{ss}F(\bx\for\by)=\rad_{C^1}F(\bx\for\by) =1$.
%;\\ $\frac{1}{2}\le\rad_{Lip}F(\bx\for\by) \le2$.
\end{itemize}
Observe that in the case of the Euclidean norm ($p=2$), the vectors in the above example, which insure that the estimate for the regularity constant is attained, satisfy also ${u^*}^Tu={v^*}^Tv=-\frac{1}{\sqrt{2}}$ and $\norm{u^\ast}_2^2+\norm{v}_2^2-({u^*}^Tu)^2 =\frac{1}{2}$.
Hence, by \eqref{rgdag} and Proposition~\ref{PropBndRadSS},
%30/03/18
$\rgdag=
{\rg=\rgo=\rgd}$.
\qed\end{example}

\if{
%\todo[inline]
%{AK 15/02/17:
%The presentation of the above example can be streamlined.}
\todo[inline]
AK 15/04/17:
The computations in the Example are relatively simple. %thanks to the employing on the space of $2\times2$ matrices the norm induced by the maximum norm in $\R^2$.
I think working with the Frobenius norm would be more challenging.

The lower bound of the radius of subregularity in the above example was established by considering only the first line of the matrix $B$, which corresponds to the first (primal space) component of the critical limit set being nonzero.
It would be good to have a
%AK 15/01/18
\red{more relevant}
example where this first component is zero, and the stability is ensured by the second one.

\AD{13/01/18.
Maybe another application/example, e.g. for a constraint system???}}\fi

%AK 7/04/18
{
\begin{example}%\label{E1}
When dealing with more complicated constraint systems than the one considered above, analyzing multiple individual cases may not be practical.
It can be more convenient to compute the needed regularity constants by solving appropriate optimization problems.
For instance, in the above example, when $p=2$ for the constant \eqref{rgdag}, we have:
\begin{multline}\label{E4-43}
(\rgdag)^2=\inf \Bigl\{\norm{u^\ast}_2^2+\norm{v}_2^2-({u^*}^Tu)^2 \mid
u^*+v^*\in\overline{N}_D(u),\\
-v^*\in\overline{N}_K(u+v),\;
{u^*}^Tu={v^*}^Tv,\;
\|u\|_2=\|v^*\|_2=1
\Bigr\}.
\end{multline}
As discussed above, when computing regularity constants, only four very similar combinations of two cases in \eqref{E4-2} and two cases in \eqref{E4-3} are of interest, and it is sufficient to consider only one of them.
For instance, the combination of the second case in \eqref{E4-2} and the second case in \eqref{E4-3} gives us $u=\left(\frac{1}{\sqrt{2}},\frac{1}{\sqrt{2}}\right)$, $v^*=(0,\pm1)$, $v\in\left\{\left(x-\frac{1}{\sqrt{2}}, -\frac{1}{\sqrt{2}}\right)\mid x\ge0\right\}$ and $u^*\in\left\{\left(-y,y\mp1\right)\mid y\ge0\right\}$.
The objective function of the respective minimization problem in the \RHS\ of \eqref{E4-43} amounts to
\begin{gather*}%\label{E4-43}
y^2+(y\mp1)^2 +\left(x-\frac{1}{\sqrt{2}}\right)^2+\frac{1}{2} -\frac{1}{2}(-y+y\mp1)^2 =x^2-\sqrt{2}x+2y^2\mp2y+\frac{3}{2},
\end{gather*}
and the compatibility constraint is fulfilled:
\begin{gather*}%\label{E4-43}
u^*{}^Tu=\frac{1}{\sqrt{2}}(-y+y\mp1)=\mp\frac{1}{\sqrt{2}}, \quad
v^*{}^Tv=(\pm1)\left(-\frac{1}{\sqrt{2}}\right) =\mp\frac{1}{\sqrt{2}}.
\end{gather*}
The respective subproblem of \eqref{E4-43} reduces, thus, to choosing the second component of the vector $v^*$: either 1 or $-1$, and two one-dimensional convex minimization problems, the second one depending on the choice:
\begin{gather*}%\label{E4-43}
\min_{x\ge0}\left(x^2-\sqrt{2}x\right)
\qdtx{and}
\min_{y\ge0}\left(y^2\mp y\right).
\end{gather*}
Since $y^2-y\le y^2+y$ for all $y\ge0$, one has to choose $v^*=(0,1)$, which leads to considering $u^*\in\left\{\left(-y,y-1\right)\mid y\ge0\right\}$ and choosing the minus sign in the second minimization problem.
The solutions $x=\frac{1}{\sqrt{2}}$ and $y=\frac{1}{2}$ of the above problems provide us with the same  ``optimal'' vectors $v=\left(0,-\frac{1}{\sqrt{2}}\right)$ and $u^*=\left(-\frac{1}{2},-\frac{1}{2}\right)$, and the value of the constant \eqref{rgdag}: $\rgdag=\frac{1}{\sqrt{2}}$.
\qed\end{example}
}

{
In general, computation of $\rgdag$ in the case of the constraint system \eqref{4.1} with Euclidean norms and polyhedral sets $D$ and $K$ amounts to solving a disjunctive program with a smooth objective function.
Computing the other regularity constants may be more demanding because of the nonsmoothness of their objective functions.
}

%AK17/05/18
At the end of the paper, we present an example, which demonstrates lack of robustness of metric subregularity.

%AK15/04/18
%\subsection{Lack of robustness}
\begin{example}%\label{E1}
Let two sequences $\{a_k\}$ and $\{b_k\}$ of positive numbers be given, such that $a_{k+1}<b_k<a_k$ $(k=1,2,\ldots)$, $a_k\to0$ (and consequently $b_k\to0$) and $\frac{b_k-a_{k+1}}{a_k-b_k}\to0$ as ${k\to+\infty}$.
For all $k=1,2,\ldots$, set
\begin{gather*}%\label{27}
\varphi(t):=
\begin{cases}
t-a_{k+1} &\mbox{if }a_{k+1}\le t<b_k,
\\
1 &\mbox{if }b_k\le t<a_k,
\end{cases}
\end{gather*}
and define a real-valued function $f$ on $(-a_1,a_1)$ by $f(x):=\int_0^{|x|}\varphi(t)dt$.
Thus, the graph of $f$ consists of linear pieces with slope 1 (when $b_k<|x|<a_k$) and parabolic pieces (when $a_{k+1}<|x|<b_k$), with the contribution of the latter diminishing as $x$ approaching 0.

Observe that $\lim_{t\uparrow b_k}\varphi(t) =b_k-a_{k+1}<1$ for all $k$ large enough, and consequently, $f(x)<|x|$ when $|x|$ is small enough.
On the other hand, $f(0)=0$ and, for any nonzero $x\in(-a_1,a_1)$ and with $n$ being the smallest natural number such that $a_n\le|x|$, we have
\begin{align*}%\label{27}
f(x)&>|x|-\sum_{k=n}^\infty(b_k-a_{k+1})
\ge|x|-\left(\max_{k\ge n} \frac{b_k-a_{k+1}}{a_k-b_k}\right) \sum_{k=n}^\infty(a_k-b_k)
\\
&>|x|-\left(\max_{k\ge n} \frac{b_k-a_{k+1}}{a_k-b_k}\right) \sum_{k=n}^\infty(a_k-a_{k+1})\ge|x|\left(1-\max_{k\ge n} \frac{b_k-a_{k+1}}{a_k-b_k}\right).
\end{align*}
Hence, $\lim_{x\to0}\frac{f(x)}{|x|}=1$, and consequently, $Df(0,0)(u)=|u|$ for all $u\in\R$.
It follows from Proposition~\ref{P3}(iv) that rg$[f](0,0)\ge1$.
(It is not difficult to show that rg$[f](0,0)=1$.)
Thus, $f$ is metrically subregular (in fact, strongly subregular) at 0 together with all its perturbations by Lipschitz continuous functions with Lipschitz modulus 1.
At the same time, $f$ is not metrically subregular at any $a_k$ $(k=1,2,\ldots)$.
\qed\end{example}

%AK4/06/18
%This means, in particular, that
Thus, for $F$ given by \eqref{4.2}, the positiveness of the radius $\rad_{Lip}F(\bx,0)$ does not imply the existence of \nbh s of $\bx$ and 0 where the subregularity is preserved.
In fact, this holds only at points $x$ and $0$ for $x$ close to $\bx$ \cite{GfrMor15}.

\section{Further research}\label{S5}

In this paper we obtain expressions and bounds for the radius of metric subregularity of mappings, in various settings, based on
generalized derivatives. In the last section we specify  these expressions/bounds for a mapping describing a system of constraints typically appearing in optimization.
We do not discuss here how to efficiently  compute these quantities; this remains an open task for further research. On a broader level, one may ask what would be the aim for having  these quantities computed.

In the Introduction we mentioned that the radius of nonsingularity of matrices is ultimately related to their condition number. The concept of conditioning plays a major role in numerical linear algebra, and preconditioning is a  highly efficient  tool for enhancing computations in numerical linear algebra. Then we come to the natural question whether
the expressions for the radius of regularity properties (not only subregularity) could be utilized in procedures for conditioning of problems of feasibility and optimization. Although there is a bulk of  studies in those directions, see the monograph \cite{BurCuc13},
the results in the whole area seem to be scattered and  lacking  unifying ideas.
 We believe that the radius theorems could serve as a basis for such a unification.
In any case, developing techniques for  conditioning of optimization problems is a challenging avenue for further research.

In this paper we consider mappings acting in  finite dimensions which is essential for the proofs. Could (some of) the results be extended to  infinite-dimensional spaces? As for the other regularity properties, there is a partial progress on that for metric regularity.   Most notably,
Ioffe constructed in \cite{Iof03.2} a Lipschitz continuous and weakly continuously Fr\'echet differentiable mapping acting in a separable Hilbert space for which the the radius equality (\ref{radreg})  is violated. On the positive side,
Ioffe and Sekiguchi \cite{IofSek09} showed that this equality  holds in infinite dimensions for certain classes of mappings with convex graphs, including  in particular semi-infinite inequality systems.

In another direction, the existing radius theorems   are quite general but cannot
be applied to situations where the perturbed mapping has  a specific form,
that is, in the case of {\em structured} perturbations; for an earlier work, see \cite{Pen05}.

For instance, there are apparently no radius theorems for the Karush-Kuhn-Tucker (KKT)
conditions in nonlinear programming, because the perturbed mapping there
ought to have the form corresponding to a KKT condition.  It is an open question
whether one might find a radius theorem, for various regularity properties,  even for the standard
nonlinear programming problem.

\section*{Acknowledgement}

The authors wish to thank the referees for their comments and suggestions.
\addcontentsline{toc}{section}{References}

\bibliographystyle{spmpsci}
\bibliography{buch-kr,kruger,kr-tmp}
\end{document}